\documentclass[12pt, twoside, fleqn, a4paper]{article}

\usepackage{latexsym}  %Extra math symbols
\usepackage{amscd}    %AMS package for doing commutative diagrams
\usepackage{theorem}  %Improved theoremlike environments
\usepackage{pifont}  %Dingbat symbols (used for end of proof symbol)
\usepackage{mathbbol} %Blackboard bold symbols
\usepackage{amsfonts}  %More fonts
\usepackage{xspace}  %Removes need for adding a \  after macros
\usepackage{amssymb} %More math symbols
\usepackage{fancyhdr} %Fancyheadings package
\usepackage{mathrsfs} %For the 3rd letters
\usepackage{amsmath} %Math symbols
\usepackage{graphics} %Picture inclusions
\usepackage{graphicx} %Picture inclusions
\usepackage{euscript}
\usepackage{psfrag}
\usepackage{empheq} %Improved equation environments 
\usepackage{mathabx} %Includes a few more stylistic arrows
\usepackage[parfill]{parskip}     %Activate to begin paragraphs with an empty line rather than an indent
\usepackage[colorlinks=true, allcolors=blue]{hyperref}

\title{A Ruelle operator for holomorphic correspondences}
\author{Shrihari Sridharan\footnote{Indian Institute of Science Education and Research Thiruvananthapuram (IISER-TVM), shrihari@iisertvm.ac.in}\ \ and\ \  Subith G.\footnote{Indian Institute of Science Education and Research Thiruvananthapuram (IISER-TVM), subith21@iisertvm.ac.in}}
\date{September 17, 2024}

\DeclareFontFamily{OT1}{pzc}{}
\DeclareFontShape{OT1}{pzc}{m}{it}%
              {<-> s * [0.900] pzcmi7t}{}
\DeclareMathAlphabet{\mathpzc}{OT1}{pzc}%
                                 {m}{it}

{\theorembodyfont{\slshape} \newtheorem{theorem}{Theorem}[section]}
{\theorembodyfont{\slshape} \newtheorem{definition}[theorem]{Definition}}
{\theorembodyfont{\slshape} \newtheorem{lemma}[theorem]{Lemma}}
{\theorembodyfont{\slshape} \newtheorem{proposition}[theorem]{Proposition}}
{\theorembodyfont{\slshape} \newtheorem{property}[theorem]{Property}}
{\theorembodyfont{\slshape} \newtheorem{corollary}[theorem]{Corollary}} 
{\theorembodyfont{\slshape} \newtheorem{claim}{Claim}} 
{\theorembodyfont{\slshape} }
{\theorembodyfont{\slshape} }

\numberwithin{equation}{section}

\newenvironment{proof}{\paragraph{Proof:}}{\hfill$\bullet$}

\begin{document}

\maketitle

\begin{abstract} 
In this paper, we extend the ideas of certain notions that one studies in thermodynamic formalism of maps to the context when the dynamics in the phase space evolves by complex holomorphic correspondences. Towards that end, we define the topological entropy of holomorphic correspondences using spanning sets. We then, define the pressure of a real-valued continuous function defined on the Riemann sphere and investigate the Ruelle operator with respect to the H\"{o}lder continuous function, however restricted on the support of the Dinh-Sibony measure. 
\end{abstract}

\begin{tabular}{l l} 
{\bf Keywords} & Entropy for a holomorphic correspondence \\ 
& Ruelle operator \\ 
& Variational Principle \\ 
& Dinh-Sibony measure \\ 
& \\ 
& \\ 
{\bf MSC Subject} & 37D35, 37A35, 37F05 \\ 
{\bf Classifications} & \\ 
\end{tabular} 
\bigskip

\section{Introduction}

Let $\widehat{\mathbb{C}}$ denote the Riemann sphere, meaning the complex plane along with the point at infinity, {\it i.e.}, $\mathbb{C} \cup \{ \infty \}$. The study of correspondences on $\widehat{\mathbb{C}}$ has seen a plethora of activity in recent years. Correspondences can be considered as a generalisation of maps. They are given by implicit relations between the variables or as a formal sum of finitely many graphs (usually more than one) on the product space. One may ascribe several properties to the considered relations, for example, continuous, holomorphic, meromorphic {\it etc.} that, in turn transfers the structure to the correspondence, thereby making its investigation interesting. Some of the recent papers that studies various aspects of the dynamics evolving through some complex correspondence include \cite{d:2005, ds:2006, ds:2008, bs:2016, bs:2021}.

Thermodynamic formalism is a branch of study in closed dynamical systems that concerns with the notions of topological entropy, pressure of real-valued continuous functions defined on the space where the dynamics takes place and the study of Ruelle operator on the appropriate space of continuous function.  In the context of rational maps restricted on their Julia sets, one defines the Ruelle operator with respect to a H\"{o}lder continuous function and results explaining the spectral radius and spectral gap are well-known in the literature, for example, \cite{DR: 1968, PW: 1975, pp:1990, mz:1996, JY: 1999, H:1999, pw:2000, LMMS:2015, cs : 2016, M:2017}. In this paper, we consider the dynamics evolving through a holomorphic correspondence and delve into certain ideas that will be useful in the context of defining and obtaining results, as one may find in thermodynamic formalism. A diligent reader may know that the concept of topological entropy was defined by Dinh and Sibony in \cite{ds:2008}, using separated sets. This was extensively studied and the entropy calculated for some certain holomorphic correspondences generated by a rational semigroup by Bharali and Sridharan in \cite{bs:2021}. Extending the ideas by Dinh and Sibony, we alternately define the topological entropy of holomorphic correspondences using spanning sets and prove that these definitions coincide. 

Further, we generalise these definitions to include a continuous function, that serves as an observable on the dynamics and enables the calculation of its pressure. A variational principle that relates the pressure of a continuous function using the integral of the function with respect to various probability measures and the topological pressure computed using separated sets and spanning sets is proved. 

Finally, we define a Ruelle operator pertaining to a continuous function, that keeps track of the dynamics of the considered holomorphic correspondence, however restricted on the support of the Dinh-Sibony measure, as defined in \cite{d:2005}. We study the spectral gap of this Ruelle operator by appealing to the spectral gap of the usual Ruelle operator. This is made possible by moving into the uncountable space of infinite orbits that emanate from the support of the Dinh-Sibony measure and remains there, a possibility of at least one orbit satisfying this property, as proved by Londhe in \cite{ml:2022}. 

The structure of this paper is as follows: In Section \eqref{ppi}, we define the central objects of this paper, namely a holomorphic correspondence and the iteration of the same. This leads us to look for permissible paths (both forward and backward) during iterations of the correspondence. In Section \eqref{tent}, we define separated sets and spanning sets in the space of orbits, recall the theorem due to Dinh and Sibony regarding the topological entropy of correspondences using separated sets and obtain the same to be equal to what we define as the topological entropy using spanning sets. In Section \eqref{prescont}, we define a quantity called pressure for any continuous function, using the ergodic sum, characterise the same and study a few properties of the pressure. In Section \eqref{trop}, we define a Ruelle operator pertaining to a continuous function, however with the restriction of the correspondence and the dynamics therein, on the support of the Dinh-Sibony measure. We state a theorem, namely Theorem \eqref{rot} analogous to the Ruelle operator theorem, in the literature in Section \eqref{expcorrsec}, for a holomorphic correspondence. The remainder of the paper is devoted to proving the Ruelle operator theorem, Theorem \eqref{rot}. 

\section{Permissible paths of iteration} 
\label{ppi}

In this section, we write some basic definitions and necessary terminologies that we will use throughout this paper. 

\begin{definition}
Let $X_{1}$ and $X_{2}$ be two compact, connected complex manifolds of dimension $n$. A holomorphic correspondence from $X_{1}$ to $X_{2}$ is a formal linear combination of the form
\begin{equation} 
\label{correspondence}
\Gamma\ \ =\ \ \sum_{1\, \leq\, j\, \leq\, N} m_{j} \Gamma_{j},
\end{equation}
where $m_{j}$'s are positive integers and $\Gamma_{1}, \Gamma_{2}, \cdots, \Gamma_{N}$ are distinct irreducible complex-analytic subvarieties of $X_{1} \times X_{2}$ of pure dimension $n$ that satisfy the following conditions:
\begin{enumerate} 
\item for each $\Gamma_{j},\ \left.\pi_{1}\right|_{\Gamma_{j}}$ and $\left.\pi_{2}\right|_{\Gamma_{j}}$ are surjective; 
\item for each $x \in X_{1}$ and $y \in X_{2},\ \left(\pi_{1}^{-1}\{x\} \cap \Gamma_{j}\right)$ and $\left(\pi_{2}^{-1}\{y\} \cap \Gamma_{j}\right)$ are finite sets for each $1 \le j \le N$, 
\end{enumerate} 
where $\pi_{i}$ is the projection onto $X_{i}$ for $i = 1, 2$. 
\end{definition}

By $\big| \Gamma \big|$, we mean the support of $\Gamma$ which is the set underlying the representation of $\Gamma$, as stated in Equation \eqref{correspondence}, {\it i.e.}, 
\[ \big| \Gamma \big|\ \ :=\ \  \left\{ (x, y) \in X_{1} \times X_{2} : (x, y) \in \Gamma_{j}\ \text{for some}\ 1 \le j \le N \right\}. \]

In this paper, we consider $X_{1} = X_{2} = \overline{\mathbb{C}}$, the Riemann sphere and call $\Gamma$ as a holomorphic correspondence on $\overline{\mathbb{C}}$. The ability to compose two correspondences introduces a perspective of dynamics to the study of correspondences. The following definition narrates a way to compose any two holomorphic correspondences $\Gamma^{1}$ and $\Gamma^{2}$ on $\overline{\mathbb{C}}$. 

\begin{definition} 
Suppose $\Gamma^{1} = \sum_{1\, \leq\, j\, \leq\, M_{1}}' \Gamma_{1, j}^{\bullet}$ and $ \Gamma^{2} = \sum_{1\, \leq\, k\, \leq\, M_{2}}' \Gamma_{2, k}^{\bullet}$ are two holomorphic correspondences on $\overline{\mathbb{C}}$, where the primed sums indicate the repetition of the varieties according to its multiplicity. Then, $\big| \Gamma^{2} \circ \Gamma^{1} \big|$ is merely the set obtained by the classical composition of $\left|\Gamma^{2}\right|$ with $\left|\Gamma^{1}\right|$ as relations. If $Y_{s, j k},\ s = 1, \cdots, M(j, k)$ are the distinct irreducible components of $\left|\Gamma_{2, k}^{\bullet}\right| \circ \left|\Gamma_{1, j}^{\bullet}\right|$, then let
\[ \eta_{s, j k}\ \ :=\ \ \# \left\{ y \in \overline{\mathbb{C}} : \text{for any}\ (x, z) \in Y_{s, j k}, \text{we have}\ (x, y) \in \Gamma_{1, j}^{\bullet}, (y, z) \in \Gamma_{2, k}^{\bullet} \right\}. \] 
This entails the following definition for composition of holomorphic correspondences: 
\[ \Gamma^{2} \circ \Gamma^{1}\ \ =\ \ \sum_{1\, \leq\, j\, \leq\, M_{1}}\ \sum_{1\, \leq\, k\, \leq\, M_{2}}\ \sum_{1\, \leq\, s\, \leq\, M(j, k)}\ \eta_{s, j k} Y_{s, j k}. \] 
\end{definition}

Suppose we consider $\Gamma$ to be a holomorphic correspondence on $\overline{\mathbb{C}}$, then the above rule for composition of correspondences yields a natural definition for iteration of $\Gamma$ with itself. For $\nu \in \mathbb{Z}_{+}$, we denote by $\Gamma^{\circ \nu}$, the $\nu$-fold composition of $\Gamma$ with itself, {\it i.e.}, $\Gamma^{\circ \nu} = \Gamma \circ \Gamma \circ \cdots \circ \Gamma$. For more details on iterative dynamics of holomorphic correspondences, readers are referred to \cite{bs:2016, d:2005}. 

We now define the set of all permissible forward paths of length $\nu \in \mathbb{Z}_{+}$ starting from an arbitrary point $x_{0} \in \overline{\mathbb{C}}$. Towards that purpose, consider ${\rm Cyl}_{\nu}$ to be the set of all $\nu$-long cylinder sets, meaning the collection of all words of length $\nu$ on the letters $\{ 1, \cdots, N \}$. Fix $\boldsymbol{\alpha} = (\alpha_{1}, \cdots, \alpha_{\nu}) \in {\rm Cyl}_{\nu}$ and define the collection of all permissible forward paths starting from $x_{0} \in \overline{\mathbb{C}}$ with respect to the combinatorial data $\boldsymbol{\alpha}$ as 
\[ \mathscr{P}_{\nu}^{\boldsymbol{\alpha}} (x_{0})\ \ =\ \ \big\{ \left( x_{0}, x_{1}, \cdots, x_{\nu}; \boldsymbol{\alpha} \right) : \left( x_{i - 1}, x_{i} \right) \in \Gamma_{\alpha_{i}} \big\}. \] 
Let $E \subseteq \overline{\mathbb{C}}$. We denote the set of all permissible forward paths of length $\nu$ starting from any point in $E$, pertaining to the combinatorial data $\boldsymbol{\alpha} \in {\rm Cyl}_{\nu}$, in $\overline{\mathbb{C}}^{\nu + 1} \times \boldsymbol{\alpha}$ as 
\[ \mathscr{P}_{\nu}^{\boldsymbol{\alpha}} (E)\ \ :=\ \ \bigcup_{x_{0}\, \in\, E} \mathscr{P}_{\nu}^{\boldsymbol{\alpha}} (x_{0}). \] 
It is then clear that $\mathscr{P}_{\nu}^{\boldsymbol{\alpha}} (x_{0}) = \mathscr{P}_{\nu}^{\boldsymbol{\alpha}} (\{ x_{0} \})$. The full set of all possible forward paths, with respect to the given holomorphic correspondence $\Gamma$, of length $\nu$ starting from $E$ is then given by 
\[ \mathscr{P}_{\nu}^{\Gamma} (E)\ \ =\ \ \bigcup_{\boldsymbol{\alpha}\, \in\, {\rm Cyl}_{\nu}}\ \bigcup_{x_{0}\, \in\, E}\ \mathscr{P}_{\nu}^{\boldsymbol{\alpha}} (x_{0}). \] 
Thus, a point $\mathfrak{X}_{\nu}^{+} (x_{0}; \boldsymbol{\alpha})  = \left( x_{0}, x_{1}, \cdots, x_{\nu}; \alpha_{1}, \cdots, \alpha_{\nu} \right) \in \mathscr{P}_{\nu}^{\boldsymbol{\alpha}} (x_{0}) \subseteq \mathscr{P}_{\nu}^{\Gamma} (x_{0}) \subseteq \mathscr{P}_{\nu}^{\Gamma} (\overline{\mathbb{C}})$ if all the points $(x_{0}, x_{j}) \in \Gamma_{\alpha_{1}} \circ \cdots \circ \Gamma_{\alpha_{j}}$ for $1 \le j \le \nu$. 

Analogously, we define a permissible backward path of length $\nu$ landing at $x_{0} \in X$, pertaining to a combinatorial data $ \boldsymbol{\beta} = (\beta_{- (\nu - 1)}, \cdots, \beta_{0}) \in {\rm Cyl}_{\nu}$ as 
\[ \mathscr{Q}_{\nu}^{\boldsymbol{\beta}} (x_{0}) = \left\{ \left( x_{- \nu}, \cdots, x_{- 1}, x_{0}; \boldsymbol{\beta} \right) : \left( x_{i - 1}, x_{i} \right) \in \Gamma_{\beta_{i}} \right\}. \] 
Similar to our definitions for permissible forward paths of length $\nu$ starting from a set $E$, we now have the collection of all permissible backward paths of length $\nu$ landing in the set $E$ as $\mathscr{Q}_{\nu}^{\boldsymbol{\beta}} (E) := \bigcup\limits_{x_{0}\, \in\, E} \mathscr{Q}_{\nu}^{\boldsymbol{\beta}} (x_{0})$. Finally, the set of all permissible backward paths of length $\nu$ for the given holomorphic correspondence $\Gamma$ landing at $x_{0}$ is given by $\mathscr{Q}_{\nu}^{\Gamma} (x_{0}) = \bigcup\limits_{\boldsymbol{\beta}\, \in\, {\rm Cyl}_{\nu}} \mathscr{Q}_{\nu}^{\boldsymbol{\beta}} (x_{0})$. Further, a point $\mathfrak{X}_{\nu}^{-} (x_{0}; \boldsymbol{\beta}) = \left( x_{- \nu}, \cdots, x_{- 1}, x_{0}; \boldsymbol{\beta} \right)$ belongs to the set $\mathscr{Q}_{\nu}^{\Gamma} (x_{0})$ if the points $(x_{- j}, x_{0}) \in \Gamma_{\beta_{-(j - 1)}} \circ \cdots \circ \Gamma_{\beta_{0}}$ for $1 \le j \le \nu$. 

Observe that the above definitions of permissible forward paths and backward paths can be naturally extended to the case where $\nu = \infty$. We call the appropriate spaces $\mathscr{P}^{\Gamma}$ and $\mathscr{Q}^{\Gamma}$ respectively, as defined below. 
\begin{eqnarray} 
\label{infiniteforwardpath} 
\mathscr{P}^{\Gamma} \left( \overline{\mathbb{C}} \right) & = & \bigcup\limits_{x_{0}\, \in\, \overline{\mathbb{C}}} \mathscr{P}^{\Gamma} \left( x_{0} \right)\ \ \ \text{and}\ \ \ \mathscr{Q}^{\Gamma} \left( \overline{\mathbb{C}} \right)\ \ =\ \ \bigcup\limits_{x_{0}\, \in\, \overline{\mathbb{C}}} \mathscr{Q}^{\Gamma} \left( x_{0} \right)\ \ \ \text{where} \nonumber \\ 
& & \nonumber \\ 
\mathscr{P}^{\Gamma} \left( x_{0} \right) & = & \left\{ \mathfrak{X}^{+} \left( x_{0}; \boldsymbol{\alpha} \right) = \left( x_{0}, x_{1}, x_{2}, \cdots; \alpha_{1}, \alpha_{2}, \cdots \right)\ : \right. \nonumber \\ 
& & \left. \hspace{+2cm} (x_{i - 1}, x_{i}) \in \Gamma_{\alpha_{i}}\ \text{and}\ \alpha_{i} \in \left\{ 1, 2, \cdots, N \right\} \right\} \\ 
& & \nonumber \\ 
\mathscr{Q}^{\Gamma} (x_{0}) & = & \left\{ \mathfrak{X}^{-} \left( x_{0}; \boldsymbol{\beta} \right) = \left( \cdots, x_{-2}, x_{-1}, x_{0}; \cdots, \beta_{-1}, \beta_{0} \right)\ : \right. \nonumber \\ 
& & \left. \hspace{+2cm} (x_{i - 1}, x_{i}) \in \Gamma_{\beta_{i}}\ \text{and}\ \beta_{i} \in \left\{ 1, 2, \cdots, N \right\} \right\}. 
\end{eqnarray} 

Henceforth, in this paper, we shall reserve the notations $\mathscr{P}$ and $\mathscr{Q}$ to denote the forward orbit and the backward orbit of appropriate length (finite or infinite), starting from or landing at a preferred point. 

We now define projection maps for each finite $j$ denoted by $\Pi_{j}^{+},\ \Pi_{j}^{-},\ {\rm proj}_{j}^{+}$ and ${\rm proj}_{j}^{-}$ from $\mathscr{P}_{\nu}^{\Gamma} (\overline{\mathbb{C}})$ and $\mathscr{Q}_{\nu}^{\Gamma} (\overline{\mathbb{C}})$ respectively to $\overline{\mathbb{C}}$ and $\left\{ 1, 2, \cdots, N \right\}^{\nu}$ by 
\begin{displaymath} 
\begin{array}{r c l c r c l l} 
\Pi_{j}^{+} (\mathfrak{X}_{\nu}^{+} (x_{0}; \boldsymbol{\alpha})) & = & x_{j} &\ \ ,\ \ & \Pi_{j}^{-} (\mathfrak{X}_{\nu}^{-} (x_{0}; \boldsymbol{\beta})) & = & x_{- j} & \text{for}\ 0 \le j \le \nu \le \infty\ \ \text{and} \\ 
\\ 
{\rm proj}_{j}^{+} (\mathfrak{X}_{\nu}^{+} (x_{0}; \boldsymbol{\alpha})) & = & \alpha_{j} & \ \ ,\ \ & {\rm proj}_{j}^{-} (\mathfrak{X}_{\nu}^{-} (x_{0}; \boldsymbol{\beta})) & = & \beta_{- j + 1} & \text{for}\ 1 \le j \le \nu \le \infty. 
\end{array} 
\end{displaymath}

\section{Topological entropy} 
\label{tent} 

Having illustrated permissible paths of forward and backward iteration of length $\nu \in \mathbb{Z}_{+}$, we now move on to define $(\nu, \epsilon)$-separated sets and $(\nu, \epsilon)$-spanning sets, for the given holomorphic correspondence $\Gamma$, as defined in Equation \eqref{correspondence}, in this section. 

 Readers may be aware that the following definition of $(\nu, \epsilon)$-separated sets was introduced by Dinh-Sibony in \cite{ds:2008}, albeit for a compact, connected, complex manifold equipped with a metric compatible with the topology therein. Since the underlying space is the Riemann sphere in our case, the metric that we bother about is the spherical metric denoted by $d_{\overline{\mathbb{C}}}$ defined on $\overline{\mathbb{C}}$. 

\begin{definition}{\cite{ds:2008}} 
\label{separated} 
Let $\Gamma$ be a holomorphic correspondence defined on $\overline{\mathbb{C}}$, as represented in Equation \eqref{correspondence} and $\mathscr{P}_{\nu}^{\Gamma} (\overline{\mathbb{C}})$ be the collection of all permissible $\nu$-orbits of $\Gamma$, as explained in Section \eqref{ppi}. Given $\epsilon > 0$, we say that any two points $\mathfrak{X}_{\nu}^{+} \left( x_{0}; \boldsymbol{\alpha}^{(1)} \right)$ and $\mathfrak{X}_{\nu}^{+} \left( y_{0}; \boldsymbol{\alpha}^{(2)} \right)$ in $\mathscr{P}_{\nu}^{\Gamma} \left( \overline{\mathbb{C}} \right)$ are $(\nu, \epsilon)$-separated if 
\begin{eqnarray*} 
\text{either} & & d_{\overline{\mathbb{C}}} \left( \Pi_{j}^{+} \left( \mathfrak{X}_{\nu}^{+} \left( x_{0}; \boldsymbol{\alpha}^{(1)} \right) \right),\ \Pi_{j}^{+} \left( \mathfrak{X}_{\nu}^{+} \left( y_{0}; \boldsymbol{\alpha}^{(2)} \right) \right) \right)\ \ >\ \ \epsilon\ \ \ \ \text{for some}\ 0 \le j \le \nu \\ 
\text{or} & & {\rm proj}_{j}^{+} \left( \mathfrak{X}_{\nu}^{+} \left( x_{0}; \boldsymbol{\alpha}^{(1)} \right) \right) \ne {\rm proj}_{j}^{+} \left( \mathfrak{X}_{\nu}^{+} \left( y_{0}; \boldsymbol{\alpha}^{(2)} \right) \right)\ \ \ \ \hspace{+1cm} \text{for some}\ 1 \le j \le \nu. 
\end{eqnarray*} 
A subset $ \mathcal{F} \subseteq \mathscr{P}_{\nu}^{\Gamma} \left( \overline{\mathbb{C}} \right)$ is said to be $(\nu, \epsilon)$-separated if any pair of distinct points in $\mathcal{F}$ are $(\nu, \epsilon)$-separated. 
\end{definition} 

As in the study of maps defined on $\overline{\mathbb{C}}$, our next goal is to define $(\nu, \epsilon)$-spanning sets in $\mathscr{P}_{\nu}^{\Gamma} \left( \overline{\mathbb{C}} \right)$. 

\begin{definition} 
\label{spanning} 
Let $\Gamma$ be a holomorphic correspondence defined on $\overline{\mathbb{C}}$, as represented in Equation \eqref{correspondence} and $\mathscr{P}_{\nu}^{\Gamma} (\overline{\mathbb{C}})$ be the collection of all permissible $\nu$-orbits of $\Gamma$, as explained in Section \eqref{ppi}. $\mathcal{G} \subseteq \mathscr{P}_{\nu}^{\Gamma} \left( \overline{\mathbb{C}} \right)$ is said to be a $(\nu, \epsilon)$-spanning set of $\mathscr{P}_{\nu}^{\Gamma} \left( \overline{\mathbb{C}} \right)$ if for every $\mathfrak{X}_{\nu}^{+} \left( x_{0}; \boldsymbol{\alpha} \right) \in \mathscr{P}_{\nu}^{\Gamma} \left( \overline{\mathbb{C}} \right)$ there exists a point $y_{0} \in \overline{\mathbb{C}}$ such that 
\[ \mathfrak{X}_{\nu}^{+} \left( y_{0}; \boldsymbol{\alpha} \right) \in \mathcal{G}\ \ \text{satisfying}\ \ d_{\overline{\mathbb{C}}} \left( \Pi_{j}^{+} \left( \mathfrak{X}_{\nu}^{+} \left( x_{0}; \boldsymbol{\alpha} \right) \right),\ \Pi_{j}^{+} \left( \mathfrak{X}_{\nu}^{+} \left( y_{0}; \boldsymbol{\alpha} \right) \right) \right)\ \ <\ \ \epsilon\ \ \ \ \text{for all}\ 0 \le j \le \nu. \] 
\end{definition} 

One primary reason for researchers working in dynamical systems to undertake a study of separated sets and spanning sets in the set of orbits is the ease with which one may define the entropy of the underlying system. The authors in \cite{ds:2008} have effectively made use of the above definition of $(\nu, \epsilon)$-separated sets to characterise the topological entropy of holomorphic correspondences. 

\begin{theorem}{\cite{ds:2008}} 
\label{T3.3 ENTROPY BY SEPARATED}
Let $\Gamma$ be a holomorphic correspondence defined on $\overline{\mathbb{C}}$, as represented in Equation \eqref{correspondence} and $\mathscr{P}_{\nu}^{\Gamma} (\overline{\mathbb{C}})$ be the collection of all permissible $\nu$-orbits of $\Gamma$. Then, the topological entropy of the holomorphic correspondence $\Gamma$ is given by 
\[ h_{{\rm top}} (\Gamma)\ \ =\ \ \sup_{\epsilon\, >\, 0}\; \limsup_{\nu\, \to\, \infty}\; \frac{1}{\nu}\; \log \left( \max \left\{ \# \mathcal{F}\ :\ \mathcal{F} \subseteq \mathscr{P}_{\nu}^{\Gamma} \left( \overline{\mathbb{C}} \right)\ \text{is a}\ (\nu, \epsilon)\text{-separated family} \right\} \right). \]
\end{theorem} 

According the product topology on the product space $\overline{\mathbb{C}}^{\nu + 1} \times \left\{ 1, 2, \cdots, N \right\}^{\nu}$, one may note that the subspace $\mathscr{P}_{\nu}^{\Gamma} (\overline{\mathbb{C}})$ is compact, since $\overline{\mathbb{C}}$ is compact and the topology compatible with the metric defined on $\left\{ 1, 2, \cdots, N \right\}^{\nu}$ should be inherited from the discrete Kronecker delta distance, that one defines on $\left\{ 1, 2, \cdots, N \right\}$ to make sense of the second condition in Definition \eqref{separated}. 

It is possible to obtain the topological entropy of holomorphic correspondences using $(\nu, \epsilon)$-spanning sets, instead of $(\nu, \epsilon)$-separated sets. In order to do the same, we now compare the optimal cardinalities of these sets, according to our needs. 

\begin{lemma} 
\label{T4.1 ENTROPY CARD OF SEP & SPAN}
Let $\epsilon > 0$ and $\nu \in \mathbb{Z}^{+}$ be given. Let 
\begin{eqnarray*} 
R_{\nu}(\epsilon) & = & \max \left\{ \# \mathcal{F} \subseteq \mathscr{P}_{\nu}^{\Gamma} (\overline{\mathbb{C}})\ :\ \mathcal{F}\ \text{is a}\ (\nu, \epsilon)\text{-separated family} \right\}; \\ 
S_{\nu}(\epsilon) & = & \min \left\{ \# \mathcal{G} \subseteq \mathscr{P}_{\nu}^{\Gamma} (\overline{\mathbb{C}})\ :\ \mathcal{G}\ \text{is a}\ (\nu, \epsilon)\text{-spanning family of}\ \mathscr{P}_{\nu}^{\Gamma} (\overline{\mathbb{C}}) \right\}. 
\end{eqnarray*} 
Then $S_{\nu}(\epsilon) \leq R_{\nu}(\epsilon) \leq S_{\nu} (\epsilon/2)$.
\end{lemma}

\begin{proof}
Let $\mathcal{F}$ be some $(\nu, \epsilon)$-separated family in $\mathscr{P}_{\nu}^{\Gamma} (\overline{\mathbb{C}})$ where the maximum cardinality is achieved, {\it i.e.}, 
\[ \mathcal{F}\ \ =\ \ \left\{ \mathfrak{X}_{\nu}^{+} \left( x^{(1)}; \boldsymbol{\alpha}^{(1)} \right),\ \mathfrak{X}_{\nu}^{+} \left( x^{(2)}; \boldsymbol{\alpha}^{(2)} \right),\ \cdots,\ \mathfrak{X}_{\nu}^{+} \left( x^{(R_{\nu}(\epsilon))} ; \boldsymbol{\alpha}^{(R_{\nu}(\epsilon))} \right) \right\}. \] 
Note that $N^{\nu} \le R_{\nu}(\epsilon)$. Now consider an arbitrary orbit $\mathfrak{X}_{\nu}^{+} \left( x; \boldsymbol{\alpha} \right) \in \mathscr{P}_{\nu}^{\Gamma} (\overline{\mathbb{C}})$. Then, $\boldsymbol{\alpha} \equiv \boldsymbol{\alpha}^{(k)}$ for some $1 \le k \le R_{\nu}(\epsilon)$. Without loss of generality, assume that $\boldsymbol{\alpha} = \boldsymbol{\alpha}^{(R_{\nu}(\epsilon))}$. For this fixed $\boldsymbol{\alpha}^{(R_{\nu}(\epsilon))}$, consider $\mathcal{F}_{\boldsymbol{\alpha}^{(R_{\nu}(\epsilon))}} = \left\{ \mathfrak{X}_{\nu}^{+} \left( \cdot; \boldsymbol{\alpha}^{(R_{\nu}(\epsilon))} \right) \in \mathcal{F} \right\}$. Then, owing to the maximality of the family $\mathcal{F}$, it is obvious that there exists some point, say $\mathfrak{X}_{\nu}^{+} \left( x_{0}, \boldsymbol{\alpha}^{(R_{\nu}(\epsilon))} \right) \in \mathcal{F}_{\boldsymbol{\alpha}^{(R_{\nu}(\epsilon))}}$ that satisfies 
\[ d_{\overline{\mathbb{C}}} \left( \Pi_{j}^{+} \left( \mathfrak{X}_{\nu}^{+} \left( x; \boldsymbol{\alpha}^{(R_{\nu}(\epsilon))} \right) \right),\ \Pi_{j}^{+} \left( \mathfrak{X}_{\nu}^{+} \left( x_{0}; \boldsymbol{\alpha}^{(R_{\nu}(\epsilon))} \right) \right) \right)\ \ \le\ \ \epsilon\ \ \ \ \text{for all}\ 0 \le j \le \nu. \] 
Thus, $\mathcal{F} \subseteq \mathscr{P}_{\nu}^{\Gamma} (\overline{\mathbb{C}})$ is a $(\nu, \epsilon)$-spanning family of $\mathscr{P}_{\nu}^{\Gamma} (\overline{\mathbb{C}})$. Hence, $S_{\nu}(\epsilon) \le R_{\nu}(\epsilon)$. 

To prove the other inequality, we now consider $\mathcal{F}$ to be a $(\nu, \epsilon)$-separated family and $\mathcal{G}$ to be a $(\nu, \epsilon/2)$-spanning family of permissible $\nu$-orbits in $\mathscr{P}_{\nu}^{\Gamma} (\overline{\mathbb{C}})$. Define a function $\tau : \mathcal{F} \longrightarrow \mathcal{G}$ as $\tau \left( \mathfrak{X}_{\nu}^{+} \left( x_{0}; \boldsymbol{\alpha} \right) \right) = \mathfrak{X}_{\nu}^{+} \left( y_{0}; \boldsymbol{\alpha} \right)$ that satisfies 
\[ d_{\overline{\mathbb{C}}} \left( \Pi_{j}^{+} \left( \mathfrak{X}_{\nu}^{+} \left( x_{0}; \boldsymbol{\alpha} \right) \right),\ \Pi_{j}^{+} \left( \mathfrak{X}_{\nu}^{+} \left( y_{0}; \boldsymbol{\alpha} \right) \right) \right) < \dfrac{\epsilon}{2}\ \ \text{for all}\ \ 0 \le j \le \nu. \] 
The existence of such an  element $\mathfrak{X}_{\nu}^{+} \left( y_{0}; \boldsymbol{\alpha} \right)$ in $\mathcal{G}$ is clear from the definition of the spanning set, Definition \eqref{spanning}. Moreover, $\tau$ is injective since suppose there exist two points, say, $\mathfrak{X}_{\nu}^{+} \left( x_{1}; \boldsymbol{\alpha} \right)$ and  $\mathfrak{X}_{\nu}^{+} \left( x_{2}; \boldsymbol{\alpha} \right) \in \mathcal{F}$ with the property that $\tau \left( \mathfrak{X}_{\nu}^{+} \left( x_{1}; \boldsymbol{\alpha} \right) \right) \equiv \tau \left( \mathfrak{X}_{\nu}^{+} \left( x_{2}; \boldsymbol{\alpha} \right) \right)$, then it contradicts our hypothesis that $\mathcal{F}$ is $(\nu, \epsilon)$-separated. Hence, $\# \mathcal{F} \le \# \mathcal{G}$. Thus, we obtain $R_{\nu}(\epsilon) \le S_{\nu} (\epsilon/2)$. 
\end{proof} 

The proof above also illustrates that for a fixed $\nu$, when $\epsilon_{1} < \epsilon_{2}$ we have $S_{\nu}(\epsilon_{1}) \ge S_{\nu}(\epsilon_{2})$. Thus, it follows directly from Theorem \eqref{T3.3 ENTROPY BY SEPARATED} and Lemma \eqref{T4.1 ENTROPY CARD OF SEP & SPAN} that the topological entropy of the correspondence $\Gamma$ can be alternately characterised using spanning sets. Thus, we have proved: 

\begin{theorem}
Let $\Gamma$ be a holomorphic correspondence defined on $\overline{\mathbb{C}}$, as represented in Equation $\eqref{correspondence}$ and $\mathscr{P}_{\nu}^{\Gamma} (\overline{\mathbb{C}})$ be the collection of all permissible $\nu$-orbits of $\Gamma$. Then if $S_{\nu}(\epsilon)$ denotes the smallest cardinality of some $(\nu, \epsilon)$-spanning family of orbits associated with $\Gamma$, then the limit $\sup\limits_{\epsilon\, >\, 0}\; \limsup\limits_{\nu\, \rightarrow\, \infty}\; \dfrac{1}{\nu}\; \log \left( S_{\nu} (\epsilon) \right)$ exists and is equal to the topological entropy of the holomorphic correspondence $\Gamma$, {\it i.e.}, 
\begin{eqnarray*} 
h_{{\rm top}} (\Gamma) & = & \sup_{\epsilon\, >\, 0}\; \limsup_{\nu\, \to\, \infty}\; \frac{1}{\nu}\; \log \left( \min \left\{ \# \mathcal{G}\ :\ \mathcal{G} \subseteq \mathscr{P}_{\nu}^{\Gamma} \left( \overline{\mathbb{C}} \right)\ \text{is a} \right. \right. \\ 
& & \left. \left. \hspace{+5cm} (\nu, \epsilon)\text{-spanning family of}\ \mathscr{P}_{\nu}^{\Gamma} \left( \overline{\mathbb{C}} \right) \right\} \right). 
\end{eqnarray*} 
\end{theorem}

\section{The pressure of continuous functions on $\overline{\mathbb{C}}$} 
\label{prescont} 

For any domain $X$, let $\mathcal{C} \left(X, \mathbb{R} \right)$ denote the set of all real-valued continuous functions defined on the domain $X$, equipped with the supremum norm given by 
\[ \left\| f \right\|_{\infty}\ \ :=\ \ \sup_{x\, \in\, X}\; \left| f(x) \right|. \] 
Suppose $T : X \longrightarrow X$ is a continuous transformation, then recall that for any function $f \in \mathcal{C} \left( X, \mathbb{R} \right)$ and for any $\nu \in \mathbb{Z}_{+}$, we define the $\nu^{{\rm th}}$ order ergodic sum of $f$ at any point $x_{0} \in X$ that moves on its trajectory under the action of $T$ as 
\begin{equation} 
\label{ergsumT} 
f_{\nu} (x_{0})\ \ =\ \ f(x_{0}) + f \left( T(x_{0}) \right) + \cdots + f \left( T^{\nu - 1} (x_{0}) \right). 
\end{equation} 
Thus, for any function $f \in \mathcal{C} \left( \overline{\mathbb{C}}, \mathbb{R} \right)$ and for any $\nu \in \mathbb{Z}_{+}$, it makes sense for us to consider the function $\mathcal{Z}_{\nu}^{+} (f)$ as an element in $\mathcal{C} \left( \mathscr{P}_{\nu}^{\Gamma} \left( \overline{\mathbb{C}} \right), \mathbb{R} \right)$ (owing to the holomorphicity of $\Gamma$) and define the $\nu^{{\rm th}}$ order ergodic sum, under the current context as 
\begin{equation}
\label{Summation}
\mathcal{Z}_{\nu}^{+} (f) \left( \mathfrak{X}_{\nu}^{+} \left( x_{0}; \boldsymbol{\alpha} \right) \right)\ \ =\ \ \sum_{0\, \le\, j\, \le\, \nu - 1} f \left( \Pi_{j}^{+} \left( \mathfrak{X}_{\nu}^{+} \left( x_{0}; \boldsymbol{\alpha} \right) \right) \right). 
\end{equation}

For any $\mathcal{Z}_{\nu}^{+} (f) \in \mathcal{C} \left( \mathscr{P}_{\nu}^{\Gamma} \left( \overline{\mathbb{C}} \right), \mathbb{R} \right)$ corresponding to $f \in \mathcal{C} \left( \overline{\mathbb{C}}, \mathbb{R} \right)$ and $\epsilon > 0$, let
\[ \mathscr{S}^{\Gamma}_{\nu}(f, \epsilon) = \inf_{\mathcal{G}} \left\{ \sum_{ \mathfrak{X}_{\nu}^{+} (x; \boldsymbol{\alpha})\, \in\, \mathcal{G}}\; e^{\mathcal{Z}_{\nu}^{+} (f) \left(\mathfrak{X}_{\nu}^{+} ( x; \boldsymbol{\alpha})\right) } \right\}, \] 
where the infimum is taken over all possible $(\nu, \epsilon)$-spanning family $\mathcal{G}$ of $\mathscr{P}_{\nu}^{\Gamma} (\overline{\mathbb{C}})$. Having defined $\mathscr{S}^{\Gamma}_{\nu}(f, \epsilon)$, we now state a few trivial properties of the same. 

\begin{lemma} 
\label{CON}
\begin{enumerate}
\item $0 < \mathscr{S}^{\Gamma}_{\nu}(f, \epsilon) \leq S_{\nu} (\epsilon) \left\| e^{\mathcal{Z}_{\nu}^{+} (f)} \right\|_{\infty} < \infty$, where $S_{\nu} (\epsilon)$ is as defined in Lemma \eqref{T4.1 ENTROPY CARD OF SEP & SPAN}. 
\item If $\epsilon_{1} < \epsilon_{2}$, then $\mathscr{S}^{\Gamma}_{\nu}\left( f, \epsilon_{1} \right) \geq \mathscr{S}^{\Gamma}_{\nu} \left(f, \epsilon_{2} \right)$.
\item $\mathscr{S}^{\Gamma}_{\nu}(0, \epsilon) = S_{\nu} (\epsilon)$, where $0$ denotes the constant function zero.
\end{enumerate}
\end{lemma} 

The properties stated in Lemma \eqref{CON} naturally prompts the following definition, that is analogous to the definition of the topological pressure for continuous functions, in the case of maps.   

\begin{definition}
Let $\Gamma$ be a holomorphic correspondence defined on $\overline{\mathbb{C}}$, as represented in Equation $\eqref{correspondence}$. For any $f \in \mathcal{C} \left( \overline{\mathbb{C}}, \mathbb{R} \right)$, we define 
\[ {\rm Pr} (\Gamma, f) = \lim_{\epsilon\, \to\, 0}\; \left( \limsup_{\nu\, \to\, \infty}\; \frac{1}{\nu} \log \mathscr{S}^{\Gamma}_{\nu} (f, \epsilon) \right). \] 
\end{definition}

The quantity inside the brackets in the above equation is monotonically increasing, as $\epsilon$ decreases and thus, ensures the existence of ${\rm Pr} (\Gamma, f)$, however, the limit can be equal to $+ \infty$. Observe that we used the set of all $(\nu, \epsilon)$-spanning families to define $\mathscr{S}^{\Gamma}_{\nu} (f, \epsilon)$, that in turn was used in defining ${\rm Pr} (\Gamma, f)$. Our next natural step is to provide a characterisation for ${\rm Pr} (\Gamma, f)$ using $(\nu, \epsilon)$-separated sets. Towards that end, we now define the following. 

For any $\mathcal{Z}_{\nu}^{+} (f) \in \mathcal{C} \left( \mathscr{P}_{\nu}^{\Gamma} \left( \overline{\mathbb{C}} \right), \mathbb{R} \right)$ corresponding to $f \in \mathcal{C} \left( \overline{\mathbb{C}}, \mathbb{R} \right)$ and $\epsilon > 0$, let
\begin{equation}
\label{separated nu eps}
\mathscr{R}^{\Gamma}_{\nu}(f, \epsilon) = \sup_{\mathcal{F}} \left\{ \sum_{ \mathfrak{X}_{\nu}^{+} (x; \boldsymbol{\alpha})\, \in\, \mathcal{F}}\; e^{\mathcal{Z}_{\nu}^{+} (f) \left(\mathfrak{X}_{\nu}^{+} ( x; \boldsymbol{\alpha})\right) } \right\}, 
\end{equation}
where the supremum is taken over all possible $(\nu, \epsilon)$-separated family in $\mathscr{P}_{\nu}^{\Gamma} (\overline{\mathbb{C}})$. Then, as earlier, we have the following lemma for $\mathscr{R}^{\Gamma}_{\nu}(f, \epsilon)$. 

\begin{lemma} 
\begin{enumerate}
\item $0 < \mathscr{R}^{\Gamma}_{\nu}(f, \epsilon) \leq R_{\nu} (\epsilon) \left\| e^{\mathcal{Z}_{\nu}^{+} (f)} \right\|_{\infty} < \infty$, where $R_{\nu} (\epsilon)$ is as defined in Lemma \eqref{T4.1 ENTROPY CARD OF SEP & SPAN}. 
\item If $\epsilon_{1} < \epsilon_{2}$, then $\mathscr{R}^{\Gamma}_{\nu}\left( f, \epsilon_{1} \right) \geq \mathscr{R}^{\Gamma}_{\nu} \left(f, \epsilon_{2} \right)$.
\item $\mathscr{R}^{\Gamma}_{\nu}(0, \epsilon) = R_{\nu} (\epsilon)$, where $0$ denotes the constant function zero.
\end{enumerate}
\end{lemma} 

Since the proof of the above lemma is elementary, we omit the same. However, we now write and prove a lemma that compares the quantities, $\mathscr{R}^{\Gamma}_{\nu} (f, \epsilon)$ and $\mathscr{S}^{\Gamma}_{\nu} (f, \epsilon)$. 

\begin{lemma} 
\label{Pressure technical theorem} 
Let $f \in \mathcal{C} (\overline{\mathbb{C}}, \mathbb{R})$. Suppose $\delta > 0$ is such that for any points $x, y \in \overline{\mathbb{C}}$ that are atmost $\epsilon/2$-apart for some $\epsilon >0$, we have $|f(x) - f(y)| < \delta$. Then, $\mathscr{R}^{\Gamma}_{\nu} (f, \epsilon) \leq e^{-\nu \delta} \mathscr{S}^{\Gamma}_{\nu} (f, \epsilon/2)$. 
\end{lemma} 

\begin{proof}
Let $\mathcal{F}$ be a $(\nu, \epsilon)$-separated family in $\mathscr{P}_{\nu}^{\Gamma} (\overline{\mathbb{C}})$ and $\mathcal{G}$, a $(\nu, \epsilon/2)$-spanning family of $\mathscr{P}_{\nu}^{\Gamma} (\overline{\mathbb{C}})$. Then, as in the proof of Lemma \eqref{T4.1 ENTROPY CARD OF SEP & SPAN}, the function $\tau : \mathcal{F} \longrightarrow \mathcal{G}$ is injective. Thus, 
\begin{eqnarray*} 
& & \sum_{\mathfrak{X}_{\nu}^{+} (y; \boldsymbol{\beta})\, \in\, \mathcal{G}} e^{\mathcal{Z}_{\nu}^{+} (f) \left(\mathfrak{X}_{\nu}^{+} ( y; \boldsymbol{\beta}) \right)} \\ 
& \ge & \sum_{\mathfrak{X}_{\nu}^{+} (y; \boldsymbol{\beta})\, \in\, \tau \left( \mathcal{F} \right)} e^{\mathcal{Z}_{\nu}^{+} (f) \left(\mathfrak{X}_{\nu}^{+} (y; \boldsymbol{\beta})\right)} \\ 
& \ge & \min_{\mathfrak{X}_{\nu}^{+} (x; \boldsymbol{\alpha})\, \in\, \mathcal{F}} \left( e^{\mathcal{Z}_{\nu}^{+} (f) \left( \tau \left( \mathfrak{X}_{\nu}^{+} (x; \boldsymbol{\alpha}) \right) \right)\; -\; \mathcal{Z}_{\nu}^{+} (f) \left( \mathfrak{X}_{\nu}^{+} (x; \boldsymbol{\alpha}) \right)} \right) \sum_{\mathfrak{X}_{\nu}^{+} (x; \boldsymbol{\alpha})\, \in\, \mathcal{F}} e^{\mathcal{Z}_{\nu}^{+} f \left( \mathfrak{X}_{\nu}^{+} (x; \boldsymbol{\alpha}) \right)} \\ 
& \ge & e^{\nu \delta}\; \times\; \sum_{\mathfrak{X}_{\nu}^{+} (x; \boldsymbol{\alpha})\, \in\, \mathcal{F}} e^{\mathcal{Z}_{\nu}^{+} (f) \left( \mathfrak{X}_{\nu}^{+} (x; \boldsymbol{\alpha}) \right)}. 
\end{eqnarray*} 
Thus, $\mathscr{S}^{\Gamma}_{\nu} (f, \epsilon/2) \geq e^{\nu \delta} \mathscr{R}^{\Gamma}_{\nu} (f, \epsilon)$.
\end{proof}

Having proved the lemma, we are now in a position to describe ${\rm Pr} (\Gamma, f)$ using $(\nu, \epsilon)$-separated sets. Recall from the proof of Lemma \eqref{T4.1 ENTROPY CARD OF SEP & SPAN} that any $(\nu, \epsilon)$-separated family with maximum cardinality in $\mathscr{P}_{\nu}^{\Gamma} (\overline{\mathbb{C}})$ must be a $(\nu, \epsilon)$-spanning family of $\mathscr{P}_{\nu}^{\Gamma} (\overline{\mathbb{C}})$. Thus, $\mathscr{S}^{\Gamma}_{\nu}(f, \epsilon) \leq \mathscr{R}^{\Gamma}_{\nu} (f, \epsilon)$. Also by Lemma \eqref{Pressure technical theorem}, we have $\mathscr{S}^{\Gamma}_{\nu} (f, \epsilon/2) \geq e^{\nu \delta} \mathscr{R}^{\Gamma}_{\nu} (f, \epsilon)$. Hence, 
\begin{eqnarray*} 
\lim_{\epsilon\, \to\, 0}\; \left( \limsup_{\nu\, \to\, \infty}\; \frac{1}{\nu} \log \mathscr{S}^{\Gamma}_{\nu} (f, \epsilon) \right) & \le & \lim_{\epsilon\, \to\, 0}\; \left( \limsup_{\nu\, \to\, \infty}\; \frac{1}{\nu} \log \mathscr{R}^{\Gamma}_{\nu} (f, \epsilon) \right) \\ 
& \le & \lim_{\epsilon\, \to\, 0}\; \left( \limsup_{\nu\, \to\, \infty}\; \frac{1}{\nu} \log \left( e^{- \nu \delta} \mathscr{S}^{\Gamma}_{\nu} (f, \epsilon/2) \right) \right) \\ 
& = & \lim_{\epsilon\, \to\, 0}\; \left( \limsup_{\nu\, \to\, \infty}\; \frac{1}{\nu} \log \mathscr{S}^{\Gamma}_{\nu} (f, \epsilon/2) \right)\; -\; \lim_{\epsilon\, \to\, 0} \delta \\ 
& = & {\rm Pr} (\Gamma, f), 
\end{eqnarray*} 
where from the statement of Lemma \eqref{Pressure technical theorem}, we know that $\delta \to 0$, as $\epsilon \to 0$, since $f \in \mathcal{C} (\overline{\mathbb{C}}, \mathbb{R})$. Thus, we have proved: 

\begin{theorem} 
\label{Pressure charac theorem}
Let $\Gamma$ be a holomorphic correspondence defined on $\overline{\mathbb{C}}$, as represented in Equation $\eqref{correspondence}$. Then, for any $f \in \mathcal{C}(\overline{\mathbb{C}}, \mathbb{R})$, 
\[ {\rm Pr} (\Gamma, f)\ \ =\ \ \lim_{\epsilon\, \to\, 0}\; \left( \limsup_{\nu\, \to\, \infty}\; \frac{1}{\nu} \log \mathscr{S}^{\Gamma}_{\nu} (f, \epsilon) \right)\ \ =\ \ \lim_{\epsilon\, \to\, 0}\; \left( \limsup_{\nu\, \to\, \infty}\; \frac{1}{\nu} \log \mathscr{R}^{\Gamma}_{\nu} (f, \epsilon) \right). \] 
\end{theorem}

It does make abundant sense to refer to the quantity ${\rm Pr} (\Gamma, f)$ as the \emph{topological pressure} of the given continuous function $f \in \mathcal{C}(\overline{\mathbb{C}}, \mathbb{R})$ with respect to the considered holomorphic correspondence $\Gamma$. We now state and verify a few standard properties of the same. 

\begin{property} 
\label{pressureofzero} 
The pressure of the zero function is indeed the topological entropy of the holomorphic correspondence $\Gamma$, {\it i.e.}, ${\rm Pr} (\Gamma, 0) = h_{{\rm top}} (\Gamma)$.
\end{property} 

\begin{property}
For any two functions $f, g \in \mathcal{C} (\overline{\mathbb{C}}, \mathbb{R})$ satisfying $f \leq g$ (pointwise), we have ${\rm Pr} (\Gamma, f) \le {\rm Pr} (\Gamma, g)$. And ${\rm Pr} (\Gamma, f + g) \leq {\rm Pr} (\Gamma, f) + {\rm Pr} (\Gamma, g)$. 
\end{property} 

Since $f \le g$, we have $\mathscr{S}^{\Gamma}_{\nu} (f, \epsilon) \le \mathscr{S}^{\Gamma}_{\nu} (g, \epsilon)$; or equivalently $\mathscr{R}^{\Gamma}_{\nu} (f, \epsilon) \le \mathscr{R}^{\Gamma}_{\nu} (g, \epsilon)$. Thus, ${\rm Pr} (\Gamma, f) \le {\rm Pr} (\Gamma, g)$. Further, $\mathscr{S}^{\Gamma}_{\nu} (f + g, \epsilon) \le \mathscr{S}^{\Gamma}_{\nu} (f, \epsilon) + \mathscr{S}^{\Gamma}_{\nu} (g, \epsilon)$ and hence, one obtains ${\rm Pr} (\Gamma, f + g) \leq {\rm Pr} (\Gamma, f) + {\rm Pr} (\Gamma, g)$. 

\begin{property} 
${\rm Pr} (\Gamma, f) \equiv \infty$ for all $f \in \mathcal{C} (\overline{\mathbb{C}}, \mathbb{R})$ or is finite-valued for all $f \in \mathcal{C} (\overline{\mathbb{C}}, \mathbb{R})$. Further, if the pressure is finite-valued then, it is a convex function that satisfies 
\[ \left| {\rm Pr} (\Gamma, f) - {\rm Pr} (\Gamma, g) \right| \leq \left\| f - g \right\|_{\infty}. \] 
\end{property} 

It is obviously true that $e^{\nu \inf f} \mathscr{S}^{\Gamma}_{\nu} (0, \epsilon) \leq \mathscr{S}^{\Gamma}_{\nu} (f, \epsilon) \leq e^{\nu \sup f} \mathscr{S}^{\Gamma}_{\nu} (0, \epsilon)$. This, in turn implies $h_{{\rm top}} (\Gamma) + \inf f \leq {\rm Pr} (\Gamma, f) \leq h_{{\rm top}} (\Gamma) + \sup f$. It is then easy to see that ${\rm Pr} (\Gamma, f)$ is equal to $\infty$ {\it iff} $h_{{\rm top}} (\Gamma) = \infty$. 

Moreover when the pressure is finite-valued, we have for any two functions $f, g \in \mathcal{C} (\overline{\mathbb{C}}, \mathbb{R})$, some $(\nu, \epsilon)$-separated set $\mathcal{F}$ in $\mathscr{P}_{\nu}^{\Gamma} (\overline{\mathbb{C}})$ and $0 \le t \le 1$, 
\[ \sum_{\mathfrak{X}_{\nu}^{+} (x, \boldsymbol{\alpha})\, \in\, \mathcal{F}} e^{\mathcal{Z}_{\nu}^{+} \left( tf + (1 - t)g \right) \left( \mathfrak{X}_{\nu}^{+} (x, \boldsymbol{\alpha}) \right)}\ \le\ \left( \sum_{\mathfrak{X}_{\nu}^{+} (x, \boldsymbol{\alpha})\, \in\, \mathcal{F}} e^{\mathcal{Z}_{\nu}^{+} \left( f \right) \left( \mathfrak{X}_{\nu}^{+} (x, \boldsymbol{\alpha}) \right)} \right)^{t} \left( \sum_{\mathfrak{X}_{\nu}^{+} (x, \boldsymbol{\alpha})\, \in\, \mathcal{F}} e^{\mathcal{Z}_{\nu}^{+} \left( g \right) \left( \mathfrak{X}_{\nu}^{+} (x, \boldsymbol{\alpha}) \right)} \right)^{1 - t} \] 
that in turn results in ${\rm Pr} (\Gamma, tf + (1 - t)g) \le t {\rm Pr} (\Gamma, f) + (1 - t) {\rm Pr} (\Gamma, g)$. Also, 
\begin{eqnarray*} 
\frac{\sup\limits_{\mathcal{F}} \sum\limits_{\mathfrak{X}_{\nu}^{+} (x, \boldsymbol{\alpha})\, \in\, \mathcal{F}} e^{\mathcal{Z}_{\nu}^{+} \left( f \right) \left( \mathfrak{X}_{\nu}^{+} (x, \boldsymbol{\alpha}) \right)}}{\sup\limits_{\mathcal{F}} \sum\limits_{\mathfrak{X}_{\nu}^{+} (x, \boldsymbol{\alpha})\, \in\, \mathcal{F}} e^{\mathcal{Z}_{\nu}^{+} \left( g \right) \left( \mathfrak{X}_{\nu}^{+} (x, \boldsymbol{\alpha}) \right)}} & \le & \sup\limits_{\mathcal{F}} \left\{ \frac{\sum\limits_{\mathfrak{X}_{\nu}^{+} (x, \boldsymbol{\alpha})\, \in\, \mathcal{F}} e^{\mathcal{Z}_{\nu}^{+} \left( f \right) \left( \mathfrak{X}_{\nu}^{+} (x, \boldsymbol{\alpha}) \right)}}{\sum\limits_{\mathfrak{X}_{\nu}^{+} (x, \boldsymbol{\alpha})\, \in\, \mathcal{F}} e^{\mathcal{Z}_{\nu}^{+} \left( g \right) \left( \mathfrak{X}_{\nu}^{+} (x, \boldsymbol{\alpha}) \right)}} \right\} \\ 
& \le & \sup\limits_{\mathcal{F}} \left\{ \max\limits_{\mathfrak{X}_{\nu}^{+} (x, \boldsymbol{\alpha})} \left\{ \frac{e^{\mathcal{Z}_{\nu}^{+} \left( f \right) \left( \mathfrak{X}_{\nu}^{+} (x, \boldsymbol{\alpha}) \right)}}{e^{\mathcal{Z}_{\nu}^{+} \left( g \right) \left( \mathfrak{X}_{\nu}^{+} (x, \boldsymbol{\alpha}) \right)}} \right\} \right\} \\ 
& \le & e^{\nu \| f - g \|_{\infty}}. 
\end{eqnarray*}
Thus, $\left| {\rm Pr} (\Gamma, f) - {\rm Pr} (\Gamma, g) \right| \leq \left\| f - g \right\|_{\infty}$. 

\begin{property} 
For any $c \in \mathbb{R}$, we have ${\rm Pr} (\Gamma, f + c) = {\rm Pr} (\Gamma, f) + c$ while ${\rm Pr} (\Gamma, cf) \leq c {\rm Pr} (\Gamma, f)$ when $c \geq 1$ and ${\rm Pr} (\Gamma, cf) \geq c {\rm Pr} (\Gamma, f)$ if $c \leq 1$.
\end{property} 

\section{Variational Principle}

Recall the definition of the set $\mathscr{P}^{\Gamma} \left( \overline{\mathbb{C}} \right)$ that contains all infinite paths from Equation \eqref{infiniteforwardpath} in Section \eqref{ppi}. We begin this section by defining a metric on $\mathscr{P}^{\Gamma} \left( \overline{\mathbb{C}} \right)$, as follows: For any two points $\mathfrak{X}^{+} \left( x_{0}; \boldsymbol{\alpha} \right) = \left( x_{0}, x_{1}, x_{2}, \cdots; \alpha_{1}, \alpha_{2}, \cdots \right)$ and $\mathfrak{X}^{+} \left( y_{0}; \boldsymbol{\beta} \right) = \left( y_{0}, y_{1}, y_{2}, \cdots; \beta_{1}, \beta_{2}, \cdots \right)$, define 
\begin{equation} 
\label{metriconinfspace} 
d_{\mathscr{P}^{\Gamma} \left( \overline{\mathbb{C}} \right)} \left( \mathfrak{X}^{+} \left( x_{0}; \boldsymbol{\alpha} \right), \mathfrak{X}^{+} \left( y_{0}; \boldsymbol{\beta} \right) \right)\ \ =\ \ \max \left\{ \sup_{\nu\, \in\, \mathbb{Z}_{+}} \frac{1}{2^{\nu}} d_{\overline{\mathbb{C}}} \left( x_{\nu}, y_{\nu} \right),\ \sup_{\nu\, \in\, \mathbb{Z}_{+}} \frac{1}{2^{\nu}} \left( 1 - \delta_{\alpha_{\nu + 1}, \beta_{\nu + 1}} \right) \right\}, 
\end{equation} 
where $\delta_{i, j}$ is the Kronecker delta function that assigns the value $1$ if $i = j$ and $0$, otherwise. In the topology dictated by the metric, one can see that $\mathscr{P}^{\Gamma} \left( \overline{\mathbb{C}} \right)$ is a compact metric space. Also consider the shift map $\sigma : \mathscr{P}^{\Gamma} \left( \overline{\mathbb{C}} \right) \longrightarrow \mathscr{P}^{\Gamma} \left( \overline{\mathbb{C}} \right)$ given by 
\begin{equation}
\label{shiftmap} 
\sigma \left( \mathfrak{X}^{+} \left( x_{0}; \boldsymbol{\alpha} \right) \right)\ \ =\ \ \sigma \left( \left( x_{0}, x_{1}, x_{2}, \cdots; \alpha_{1}, \alpha_{2}, \cdots \right) \right)\ \ =\ \ \left( x_{1}, x_{2}, x_{3}, \cdots; \alpha_{2}, \alpha_{3}, \cdots \right) 
\end{equation}

Since $\sigma$ is a continuous transformation defined on the compact metric space $\mathscr{P}^{\Gamma} \left( \overline{\mathbb{C}} \right)$, one can now write the definition of the topological pressure of a continuous function $F \in \mathcal{C} \left( \mathscr{P}^{\Gamma} \left( \overline{\mathbb{C}} \right), \mathbb{R} \right)$ with respect to the shift map $\sigma$, in accordance to Bowen, \cite{pw:1982}. 

\begin{definition} 
For any $\nu \in \mathbb{Z}_{+}$ and $\epsilon > 0$, we say that $\mathcal{F} \subseteq \mathscr{P}^{\Gamma} \left( \overline{\mathbb{C}} \right)$ is a $(\nu, \epsilon)$-separated family with respect to the shift map $\sigma$ if for any two points $\mathfrak{X}^{+} \left( x_{0}; \boldsymbol{\alpha} \right)$ and $\mathfrak{X}^{+} \left( y_{0}; \boldsymbol{\beta} \right)$, we have 
\[ \max\limits_{0\, \le\, j\, \le\, \nu -1} \left\{ d_{\mathscr{P}^{\Gamma} \left( \overline{\mathbb{C}} \right)} \left( \sigma^{j} \left( \mathfrak{X}^{+} \left( x_{0}; \boldsymbol{\alpha} \right) \right), \sigma^{j} \left( \mathfrak{X}^{+} \left( y_{0}; \boldsymbol{\beta} \right) \right) \right) \right\} > \epsilon. \] 
Then, for any $F \in \mathcal{C} \left( \mathscr{P}^{\Gamma} \left( \overline{\mathbb{C}} \right), \mathbb{R} \right)$, its topological pressure with respect to the map $\sigma$ is given by 
\begin{equation} 
\label{classical pressure} 
{\rm Pr}_{\sigma} (F)\ \ =\ \ \lim_{\epsilon\, \to\, 0}\; \limsup_{\nu\, \to\, \infty}\; \frac{1}{\nu}\; \log \left( \sup_{\mathcal{F}} \left\{ \sum_{\mathfrak{X}^{+} \left( x_{0}; \boldsymbol{\alpha} \right)\, \in\, \mathcal{F}} e^{F_{\nu} \left( \mathfrak{X}^{+} \left( x_{0}; \boldsymbol{\alpha} \right) \right)} \right\} \right), 
\end{equation} 
where the supremum is taken over all $(\nu, \epsilon)$-separated family $\mathcal{F}$ in $\mathscr{P}^{\Gamma} \left( \overline{\mathbb{C}} \right)$. Further, the quantity $F_{\nu} \left( \mathfrak{X}^{+} \left( x_{0}; \boldsymbol{\alpha} \right) \right)$ is the $\nu^{{\rm th}}$ order ergodic sum of $F$ evaluated at the point $\mathfrak{X}^{+} \left( x_{0}; \boldsymbol{\alpha} \right) \in \mathscr{P}^{\Gamma} \left( \overline{\mathbb{C}} \right)$ under the action of $\sigma$, as in Equation \eqref{ergsumT} and is given by 
\begin{equation} 
\label{ergodsum} 
F_{\nu} \left( \mathfrak{X}^{+} \left( x_{0}; \boldsymbol{\alpha} \right) \right)\ \ =\ \ \sum\limits_{0\, \le\, j\, \le\, \nu -1} F \left( \sigma^{j} \left( \mathfrak{X}^{+} \left( x_{0}; \boldsymbol{\alpha} \right) \right) \right). 
\end{equation} 
\end{definition} 

The next theorem relates the topological pressure of a continuous function defined on $\overline{\mathbb{C}}$ with respect to the holomorphic correspondence $\Gamma$ to its counterpart defined on $\mathscr{P}^{\Gamma} \left( \overline{\mathbb{C}} \right)$ with respect to the shift map $\sigma$. 

\begin{theorem}
\label{variational pre}
For every continuous function $f \in \mathcal{C} \left( \overline{\mathbb{C}}, \mathbb{R} \right)$, we have ${\rm Pr} (\Gamma, f) \equiv {\rm Pr}_{\sigma} (f \circ \Pi_{0}^{+})$, where $\Pi_{0}^{+}$ is the projection map, as defined in Section \eqref{ppi}. 
\end{theorem}

\begin{proof}
For any function $f \in  \mathcal{C} \left( \overline{\mathbb{C}}, \mathbb{R} \right)$, it is obvious that $F = f \circ \Pi_{0}^{+} \in \mathcal{C} \left( \mathscr{P}^{\Gamma} \left( \overline{\mathbb{C}} \right), \mathbb{R} \right)$. Consider any point $\mathfrak{X}_{\nu}^{+} \left( x_{0}; \boldsymbol{\alpha}^{\prime} \right) \in \mathscr{P}_{\nu} ^{\Gamma} (\overline{\mathbb{C}})$. Then, there exists a corresponding point (not necessarily unique) $\mathfrak{X}^{+} \left( x_{0}; \boldsymbol{\alpha} \right) \in \mathscr{P}^{\Gamma} \left( \overline{\mathbb{C}} \right)$ such that $\Pi_{j}^{+} \left( \mathfrak{X}^{+} \left( x_{0}; \boldsymbol{\alpha} \right) \right) = \Pi_{j}^{+} \left( \mathfrak{X}_{\nu}^{+} \left( x_{0}; \boldsymbol{\alpha}^{\prime} \right) \right)$ for $0 \le j \le \nu$ and 
${\rm proj}_{j}^{+} \left( \mathfrak{X}^{+} \left( x_{0}; \boldsymbol{\alpha} \right) \right) = {\rm proj}_{j}^{+} \left( \mathfrak{X}_{\nu}^{+} \left( x_{0}; \boldsymbol{\alpha}^{\prime} \right) \right)$ for $1 \le j \le \nu$. Moreover, $\mathcal{Z}_{\nu}^{+} (f) \left( \mathfrak{X}_{\nu}^{+} \left( x_{0}; \boldsymbol{\alpha}^{\prime} \right) \right) = F_{\nu} \left( \mathfrak{X}_{\nu}^{+} \left( x_{0}; \boldsymbol{\alpha} \right) \right)$ for all $\nu \in \mathbb{Z}_{+}$. 

Now fix $\epsilon > 0$ and $\nu \in \mathbb{Z}_{+}$. Recall from Equation \eqref{separated nu eps} that 
\[ \mathscr{R}^{\Gamma}_{\nu} \left( f, \epsilon \right) = \sup \left\{ \sum_{\mathfrak{X}_{\nu}^{+} \left( x; \boldsymbol{\alpha} \right)\, \in\, \mathcal{F}} e^{ \mathcal{Z}_{\nu}^{+} (f) \left( \mathfrak{X}_{\nu}^{+} \left( x; \boldsymbol{\alpha} \right) \right)}\ \mid\ \mathcal{F}\ \text{is a}\ (\nu, \epsilon)\text{-separated family of}\ \mathscr{P}_{\nu}^{\Gamma} \left( \overline{\mathbb{C}} \right) \right\}. \] 
Exploiting the same idea, one can then define 
\[ R_{\nu}^{\sigma} \left( F, \epsilon \right) = \sup \left\{ \sum_{\mathfrak{X}^{+} \left( x; \boldsymbol{\alpha} \right)\, \in\, \mathcal{F}} e^{ F_{\nu} \left( \mathfrak{X}^{+} \left( x; \boldsymbol{\alpha} \right) \right)}\ \mid\ \mathcal{F}\ \text{is a}\ (\nu, \epsilon)\text{-separated family of}\ \mathscr{P}^{\Gamma} \left( \overline{\mathbb{C}} \right) \right\}. \] 

Let $\mathfrak{X}_{\nu}^{+} \left( x_{0}; \boldsymbol{\alpha}^{\prime} \right), \mathfrak{X}_{\nu}^{+} \left( y_{0}; \boldsymbol{\beta}^{\prime} \right) \in \mathcal{F}$, a $(\nu, \epsilon)$ separated subset of $\mathscr{P}_{\nu} ^{\Gamma} (\overline{\mathbb{C}})$. Then, we know that there exists corresponding points (not necessarily unique) $\mathfrak{X}^{+} \left( x_{0}; \boldsymbol{\alpha} \right), \mathfrak{X}^{+} \left( y_{0}; \boldsymbol{\beta} \right) \in \mathscr{P}^{\Gamma} \left( \overline{\mathbb{C}} \right)$ such that either $d_{\overline{\mathbb{C}}} \left( x_{j}, y_{j} \right) > \epsilon$ for some $0 \le j \le \nu$ or $\alpha_{j} \ne \beta_{j}$ for some $1 \le j \le \nu$. Thus, 
\[ \max\limits_{0\, \le\, j\, \le\, \nu - 1} \left\{ d_{\mathscr{P}^{\Gamma} \left( \overline{\mathbb{C}} \right)} \left( \sigma^{j} \left( \mathfrak{X}^{+} \left( x_{0}; \boldsymbol{\alpha} \right) \right),\; \sigma^{j} \left( \mathfrak{X}^{+} \left( y_{0}; \boldsymbol{\beta} \right) \right) \right) \right\}\ \ \ge\ \ \epsilon, \] 
meaning the points $\mathfrak{X}^{+} \left( x_{0}; \boldsymbol{\alpha} \right)$ and $\mathfrak{X}^{+} \left( y_{0}; \boldsymbol{\beta} \right)$ belong to some $(\nu, \epsilon)$-separated set of $\mathscr{P}^{\Gamma} \left( \overline{\mathbb{C}} \right)$. Moreover, since $\mathcal{Z}_{\nu}^{+} (f) \left( \mathfrak{X}_{\nu}^{+} \left( x_{0}; \boldsymbol{\alpha}^{\prime} \right) \right) = F_{\nu} \left( \mathfrak{X}_{\nu}^{+} \left( x_{0}; \boldsymbol{\alpha} \right) \right)$ for all $\nu \in \mathbb{Z}_{+}$, we have the inequality  $\mathscr{R}_{\nu}^{\Gamma} \left( f, \epsilon \right) \le R_{\nu}^{\sigma} \left( F, \epsilon \right)$ for all $\nu \in \mathbb{Z}_{+}$. 

To prove the other way inequality, we again fix an $\epsilon > 0$ and $\nu \in \mathbb{Z}_{+}$. Let $\mathcal{F}$ be a $(\nu, \epsilon)$-separated set of $\mathscr{P}^{\Gamma} \left( \overline{\mathbb{C}} \right)$. Then, there exists a $N(\epsilon) \in \mathbb{Z}_{+}$ such that $\dfrac{1}{2^{m}} < \epsilon$ for all $m > N(\epsilon)$. For any two elements say, $\mathfrak{X}^{+} \left( x_{0}; \boldsymbol{\alpha} \right), \mathfrak{X}^{+} \left( y_{0}; \boldsymbol{\beta} \right) \in \mathcal{F} \subseteq \mathscr{P}^{\Gamma} \left( \overline{\mathbb{C}} \right)$, we have 
\begin{equation} 
\label{max} 
\max_{0\, \le\, j\, \le\, \nu - 1} d_{\mathscr{P}^{\Gamma} \left( \overline{\mathbb{C}} \right)} \left( \sigma^{j} \left( \mathfrak{X}^{+} \left( x_{0}; \boldsymbol{\alpha} \right) \right), \sigma^{j} \left( \mathfrak{X}^{+} \left( y_{0}; \boldsymbol{\beta} \right) \right) \right)\ \ >\ \ \epsilon. 
\end{equation} 
Equation \eqref{max} in turn implies at least one of the following statements to be true. 
\begin{enumerate}
\item Either $d_{\overline{\mathbb{C}}} \left( x_{j}, y_{j} \right) > \epsilon$ for some $0 \le j \le \nu - 1 + N(\epsilon)$, 
\item Or $\alpha_{j} \ne \beta_{j}$ for some $1 \le j \le \nu - 1 + N(\epsilon)$. 
\end{enumerate}
Hence, truncating the points $\mathfrak{X}^{+} \left( x_{0}; \boldsymbol{\alpha} \right)$ and $\mathfrak{X}^{+} \left( y_{0}; \boldsymbol{\beta} \right)$ that describes the infinitely long trajectories that start at $x_{0}$ and $y_{0}$ respectively, as necessary, we can find $\mathfrak{X}_{\nu + N(\epsilon)}^{+} \left( x_{0}; \boldsymbol{\alpha}^{\prime} \right)$ and $\mathfrak{X}_{\nu + N(\epsilon)}^{+} \left( y_{0}; \boldsymbol{\beta}^{\prime} \right) \in \mathscr{P}_{\nu + N(\epsilon)}^{\Gamma} \left( \overline{\mathbb{C}} \right)$, {\it i.e.}, the points $\mathfrak{X}_{\nu + N(\epsilon)}^{+} \left( x_{0}; \boldsymbol{\alpha}^{\prime} \right)$ and $\mathfrak{X}_{\nu + N(\epsilon)}^{+} \left( y_{0}; \boldsymbol{\beta}^{\prime} \right)$ are precisely those for which 
\begin{displaymath} 
\begin{array}{r c l l} 
\Pi_{j}^{+} \left( \mathfrak{X}^{+} \left( x_{0}; \boldsymbol{\alpha} \right) \right) & \equiv & \Pi_{j}^{+} \left( \mathfrak{X}_{\nu + N(\epsilon)}^{+} \left( x_{0}; \boldsymbol{\alpha}^{\prime} \right) \right) & \text{and} \\ 
\Pi_{j}^{+} \left( \mathfrak{X}^{+} \left( y_{0}; \boldsymbol{\beta} \right) \right) & \equiv & \Pi_{j}^{+} \left( \mathfrak{X}_{\nu + N(\epsilon)}^{+} \left( y_{0}; \boldsymbol{\beta}^{\prime} \right) \right), & \text{for}\ 0 \le j \le \nu + N(\epsilon) - 1, \\ 
\\ 
{\rm proj}_{j}^{+} \left( \mathfrak{X}^{+} \left( x_{0}; \boldsymbol{\alpha} \right) \right) & \equiv & {\rm proj}_{j}^{+} \left( \mathfrak{X}_{\nu + N(\epsilon)}^{+} \left( x_{0}; \boldsymbol{\alpha}^{\prime} \right) \right) & \text{and} \\ 
{\rm proj}_{j}^{+} \left( \mathfrak{X}^{+} \left( y_{0}; \boldsymbol{\beta} \right) \right) & \equiv & {\rm proj}_{j}^{+} \left( \mathfrak{X}_{\nu + N(\epsilon)}^{+} \left( y_{0}; \boldsymbol{\beta}^{\prime} \right) \right), & \text{for}\ 1 \le j \le \nu + N(\epsilon) - 1. 
\end{array} 
\end{displaymath} 
In other words, the points $\mathfrak{X}_{\nu + N(\epsilon)}^{+} \left( x_{0}; \boldsymbol{\alpha}^{\prime} \right)$ and $\mathfrak{X}_{\nu + N(\epsilon)}^{+} \left( y_{0}; \boldsymbol{\beta}^{\prime} \right)$ belong to some $(\nu + N(\epsilon), \epsilon)$-separated family in $\mathscr{P}_{\nu + N(\epsilon)}^{\Gamma} \left( \overline{\mathbb{C}} \right)$. Thus, $\mathscr{R}_{\nu + N(\epsilon)}^{\Gamma} \left( f, \epsilon \right) \ge R_{\nu + N(\epsilon)}^{\sigma} \left( F, \epsilon \right)$ for all $\nu \in \mathbb{Z}_{+}$. 

Since the definition of the pressure of $f$, as stated in Theorem \eqref{Pressure charac theorem} and the pressure of $F$, as stated in Equation \eqref{classical pressure} involves the limiting values of $\mathscr{R}_{\nu}^{\Gamma} \left( f, \epsilon \right)$ and $R_{\nu}^{\sigma} \left( F, \epsilon \right)$, as $\nu \to \infty$, we conclude that ${\rm Pr} (\Gamma, f) = {\rm Pr}_{\sigma} (F)$ where $F = f \circ \Pi_{0}^{+}$. 
\end{proof} 

As a corollary to Theorem \eqref{variational pre}, one can write the variational principle for a holomorphic correspondence as follows: 

\begin{corollary} (Variational principle for a holomorphic correspondence)
\label{varprin}
Let $\Gamma$ be the holomorphic correspondence defined on the Riemann sphere $\overline{\mathbb{C}}$, as written in Definition \eqref{correspondence}. Then, for any $f \in \mathcal{C} \left( \overline{\mathbb{C}}, \mathbb{R} \right)$, we have 
\[ {\rm Pr} (\Gamma, f)\ =\ \sup_{\mu} \left\{ h_{\mu} (\sigma) + \int F \mathrm{d}\mu \right\}\ =\ \sup_{\mu} \left\{ h_{\mu} (\sigma) + \int f \circ \Pi_{0}^{+} \mathrm{d}\mu \right\}, \] 
where the supremum is taken over all $\sigma$-invariant probability measures $\mu$ supported on $\mathscr{P}^{\Gamma} \left( \overline{\mathbb{C}} \right)$ and $h_{\mu} (\sigma)$ is the measure-theoretic entropy of the map $\sigma$, associated to the measure $\mu$. 
\end{corollary} 

Further, we obtain a Theorem stated in \cite{bs:2021} that relates the topological entropy of the holomorphic correspondence $\Gamma$ and the shift map $\sigma$, as a corollary to the variational principle for holomorphic correspondence, as written in Corollary \eqref{varprin}, Theorem \eqref{variational pre} and Property \eqref{pressureofzero}. 

\begin{corollary} 
\[ h_{{\rm top}} (\Gamma)\ \ =\ \ {\rm Pr} (\Gamma, 0)\ \ =\ \ {\rm Pr}_{\sigma} \left( 0 \circ \Pi_{0}^{+} \right)\ \ =\ \ h(\sigma). \]
\end{corollary} 

\section{The Ruelle operator} 
\label{trop} 

For a holomorphic correspondence $\Gamma$ on $\overline{\mathbb{C}}$, as stated in Equation \eqref{correspondence}, we define 
\begin{equation} 
\label{done} 
d_{1} (y)\ \ =\ \ \sum_{j\, =\, 1}^{N} m_{j}\; {\rm Card}\left\{ x \in \overline{\mathbb{C}} : (x, y) \in \Gamma_{j} \right\}\ \ \ \text{for any}\ y \in \overline{\mathbb{C}}. 
\end{equation} 
Let $\Gamma$ be a correspondence that satisfies $d_{1} (y) = d_{1}$, for generic points $y \in \overline{\mathbb{C}}$. Observe that this definition of $d_{1}$ extends from what we know as the topological degree for maps, that counts the number of pre-images for generic points. Further, given a holomorphic correspondence $\Gamma$ on $\overline{\mathbb{C}}$, we define its adjoint holomorphic correspondence $^{\dagger}\Gamma$ on $\overline{\mathbb{C}}$ as 
\[ ^{\dagger}\Gamma\ \ =\ \ \sum_{1\, \leq\, j\, \leq\, N} m_{j} ^{\dagger}\Gamma_{j}, \] 
where $^{\dagger}\Gamma_{j} = \left\{ (x, y) \in \overline{\mathbb{C}} \times \overline{\mathbb{C}} : (y, x) \in \Gamma_{j} \right\}$. We define $d_{0}$ for the correspondence $\Gamma$ as $d_{0} = d_{1} (^{\dagger}\Gamma)$. From now on, we work with holomorphic correspondences on $\overline{\mathbb{C}}$ that satisfy $d_{0} < d_{1}$, a necessary condition for the following theorem to hold. 

\begin{theorem} \cite{bs:2016} 
Let $\Gamma$ be a holomorphc correspondence on $\overline{\mathbb{C}}$, as written in Equation \eqref{correspondence} satisfying the condition $d_{1} > d_{0}$. Suppose $x \in \overline{\mathbb{C}}\; \setminus\; \mathcal{E}$, where $\mathcal{E}$ is some polar set. Then, the sequence of measures $\left\{ \dfrac{1}{d_{1}^{\nu}} \sum\limits_{y\, \in\, X\; :\; (y, x)\, \in\, \Gamma^{\circ \nu}}\ \delta_{y} \right\}_{\nu\, \ge\, 1}$ converges (in the weak* topology) to some measure, independent of $x$. 
\end{theorem} 

This limiting measure is called the \emph{Dinh-Sibony measure}, that we denote in this paper by $\mu$ and the support of $\mu$, denoted by $\mathcal{X}$ remains away from the normality set of $\Gamma$, as proved in \cite{bs:2016}. Further, $\mathcal{X}$ is backward invariant under the action of $\Gamma$, {\it i.e.}, for any $x \in \mathcal{X}$, 
\[ \bigcup_{\nu\, \ge\, 1} \big\{ y \in \overline{\mathbb{C}}\ :\ (y, x) \in \Gamma^{\circ \nu} \big\}\ \ \subseteq\ \ \mathcal{X}. \] 

As earlier, let $\mathcal{C}(\mathcal{X}, \mathbb{R})$ denote the Banach space of real-valued continuous functions defined on the compact metric space $\mathcal{X}$ equipped with the essential supremum norm. For any fixed $f \in \mathcal{C}(\mathcal{X}, \mathbb{R})$, we define a Ruelle operator on $\mathcal{C}(\mathcal{X}, \mathbb{R})$ denoted by $\mathcal{L}_{f}$ whose action on the points in $\mathcal{X}$ is as prescribed below.  
\[ \left( \mathcal{L}_{f} (g) \right) (x)\ \ :=\ \ \sum_{y\, \in\, \left\{ (y, x)\, \in\, \Gamma \right\}} e^{f(y)} g(y). \] 
Since the action of $\mathcal{L}_{f}( g )$ on points in $\mathcal{X}$ depends continuously on $\mathcal{X}$, it is obvious that $\mathcal{L}_{f} (g) \in \mathcal{C} \left( \mathcal{X}, \mathbb{R} \right)$. Moreover, $\mathcal{L}_{f}$ is a bounded linear operator. Note that the definition of the Ruelle operator naturally leads us to the iterates of the Ruelle operator, namely $\mathcal{L}_{f}^{\nu}$ for any $\nu \in \mathbb{Z}_{+}$, that satisfies, 
\[ \left( \mathcal{L}_{f}^{\nu} (g) \right) (x)\ \ =\ \ \sum_{\boldsymbol{\beta}\, \in\, {\rm Cyl}_{\nu}}\ \ \sum_{\mathfrak{X}_{\nu}^{-} \left( x; \boldsymbol{\beta} \right)\, \in\, \mathscr{Q}_{\nu}^{\Gamma} (x)} e^{\sum_{j\, =\, 1}^{\nu} f \left( \Pi_{j}^{-} \left( \mathfrak{X}_{\nu}^{-} \left( x; \boldsymbol{\beta} \right) \right) \right)} g \left( \Pi_{\nu}^{-} \left( \mathfrak{X}_{\nu}^{-} \left( x; \boldsymbol{\beta} \right) \right) \right). \]
For ease of writing and keeping in tune with the definition of $\nu$-th order ergodic sum, as stated in Equation \eqref{Summation}, we define and denote 
\[ \sum_{j\, =\, 1}^{\nu} f \left( \Pi_{j}^{-} \left( \mathfrak{X}_{\nu}^{-} \left( x; \boldsymbol{\beta} \right) \right) \right)\ \ =\ \ \mathcal{Z}_{\nu}^{-} (f) \left( \mathfrak{X}_{\nu}^{-} \left( x; \boldsymbol{\beta} \right) \right), \] 
so that the action of the iterates of the Ruelle operator can be rewritten as 
\[ \left( \mathcal{L}_{f}^{\nu} (g) \right) (x)\ \ :=\ \ \sum_{\boldsymbol{\beta}\, \in\, {\rm Cyl}_{\nu}}\ \ \sum_{\mathfrak{X}_{\nu}^{-} \left( x; \boldsymbol{\beta} \right)\, \in\, \mathscr{Q}_{\nu}^{\Gamma} (x)} e^{\mathcal{Z}_{\nu}^{-} (f) \left( \mathfrak{X}_{\nu}^{-} \left( x; \boldsymbol{\beta} \right) \right)} g \left( \Pi_{\nu}^{-} \left( \mathfrak{X}_{\nu}^{-} \left( x; \boldsymbol{\beta} \right) \right) \right). \] 

Let $\mathcal{L}_{f}^{*}$ denote the adjoint operator, corresponding to the Ruelle operator $\mathcal{L}_{f}$, defined on the dual space of $\mathcal{C} \left( \mathcal{X}, \mathbb{R} \right)$, namely the space of all signed measures supported on $\mathcal{X}$, denoted by $\mathcal{C} \left( \mathcal{X}, \mathbb{R} \right)^{*}$. For every signed measure $\rho \in \mathcal{C} \left( \mathcal{X}, \mathbb{R} \right)^{*}$, the action of the measure $\mathcal{L}_{f}^{*} (\rho)$ on any $g \in \mathcal{C} \left( \mathcal{X}, \mathbb{R} \right)$ is given by 
\[ \left[ \mathcal{L}_{f}^{*} (\rho) \right] (g)\ \ =\ \ \rho \left( \mathcal{L}_{f} (g) \right)\ \ =\ \ \int_{\mathcal{X}} \left( \mathcal{L}_{f} g \right) \mathrm{d} \rho. \] 
Observe that the set of probability measures supported on $\mathcal{X}$, denoted by $\mathcal{M} (\mathcal{X})$ is a subset of $\mathcal{C} \left( \mathcal{X}, \mathbb{R} \right)^{*}$. 

\section{Expansive correspondences} 
\label{expcorrsec} 

As in the case of dynamics of maps, we expect the holomorphic correspondence $\Gamma$ to satisfy a condition of expansion. However, since the action of the correspondence on $\mathcal{X}$ is not forward invariant, we present a restrictive definition of expansivity, in order that we consider only those points that always remain in $\mathcal{X}$. 

\begin{definition} 
\label{expcorr} 
A holomorphic correspondence $\Gamma$, as represented in Equation \eqref{correspondence}, is said to be an \emph{expansive correspondence on $\mathcal{X}$} if there exists a constant $\lambda > 1$ such that for any pair of points $x, y \in \mathcal{X}$ with $(x_{-1}, x) \in \Gamma_{j}$, a variety in the definition of $\Gamma$, there exists a point $y_{-1} \in \mathcal{X}$ such that 
\begin{enumerate} 
\item $(y_{-1}, y)$ belongs to the same variety $\Gamma_{j}$ as $(x_{-1}, x)$ and 
\item $\lambda d_{\overline{\mathbb{C}}} (x_{-1}, y_{-1}) < d_{\overline{\mathbb{C}}} (x, y)$. 
\end{enumerate} 
\end{definition} 

From now on, we shall consider expansive holomorphic correspondences $\Gamma$ restricted on the support of the Dinh-Sibony measure, $\mathcal{X}$, a compact set in $\overline{\mathbb{C}}$. For every $f \in \mathcal{C} (\mathcal{X})$ and $k \in \mathbb{Z}_{+}$, let 
\begin{equation} 
\label{omegak} 
\omega_{k} (f)\ \ :=\ \ \sup \left\{ \left| f(x) - f(y) \right|\ :\ d_{\overline{\mathbb{C}}} (x, y) \le \frac{1}{\lambda^{k - 1}} \right\} 
\end{equation} 
and consider the subset $\mathcal{C}^{\alpha} \left( \mathcal{X}, \mathbb{R} \right) \subset \mathcal{C} \left( \mathcal{X}, \mathbb{R} \right)$ where $\alpha = \lambda^{-1}$ given by  
\[ \mathcal{C}^{\alpha} \left( \mathcal{X}, \mathbb{R} \right)\ \ =\ \ \left\{ f \in \mathcal{C} \left( \mathcal{X}, \mathbb{R} \right) : \sum_{k \ge 1} \omega_{k} (f) + \| f \|_{\infty} < \infty \right\}. \] 
Then, it is a well-known fact that $\mathcal{C}^{\alpha} \left( \mathcal{X}, \mathbb{R} \right)$ is a Banach space with respect to the norm, 
\[ \| f \|_{\alpha}\ \ =\ \ \sum_{k\, \ge\, 1} \omega_{k} (f) + \| f \|_{\infty}. \] 

We now state the main result of this section, namely the Ruelle operator theorem, pertaining to expansive holomorphic correspondences.

\begin{theorem} 
\label{rot} 
Let $\Gamma$, as represented in Equation \eqref{correspondence}, denote an expansive holomorphic correspondence and $\mathcal{X}$, the support of the Dinh-Sibony measure associated to $\Gamma$. Consider the Ruelle operator $\mathcal{L}_{f}$ corresponding to a function $f \in \mathcal{C}^{\alpha} \left( \mathcal{X}, \mathbb{R} \right)$, defined on $\mathcal{C} \left( \mathcal{X}, \mathbb{R} \right)$. Then, 
\begin{enumerate} 
\item There exists a simple, positive eigenvalue, say $\Lambda$ of $\mathcal{L}_{f}$ with a corresponding positive eigenfunction, say $h \in \mathcal{C} \left( \mathcal{X}, \mathbb{R} \right)$. 
%\item $\Lambda$ is a simple, positive eigenvalue of $\mathcal{L}_{f}^{*}$ with corresponding eigenmeasure, say $m \in \mathcal{M} (\mathcal{X})$. 
\item For every $g \in \mathcal{C} \left( \mathcal{X}, \mathbb{R} \right)$, the sequence $\left\{ \dfrac{1}{\Lambda^{\nu}} \mathcal{L}_{f}^{\nu} (g) \right\}_{\nu\, \ge\, 1}$ converges uniformly in $\mathcal{X}$. 
%\item $\log \Lambda = {\rm Pr}(f)$. 
\end{enumerate} 
\end{theorem} 

We now write parts of the proof of Theorem \eqref{rot}, by making the following claims and proving them one-by-one. 

\begin{claim} 
There exists some $m \in \mathcal{M} (\mathcal{X})$ such that $\mathcal{L}_{f}^{*} m = \Lambda m$. 
\end{claim} 

\begin{proof} 
Define a map $\Phi$ on $\mathcal{M} (\mathcal{X})$ by 
\[ \displaystyle{\Phi (\rho)\ \ :=\ \ \frac{\mathcal{L}_{f}^{*} (\rho)}{\int_{\mathcal{X}} \mathcal{L}_{f} (1) \mathrm{d}\rho}}. \] 
Note that $\displaystyle{\int_{\mathcal{X}} \mathcal{L}_{f}(1) \mathrm{d}\rho} > 0$, thus making the map $\Phi$ continuous. Since $\mathcal{X}$ is a compact metric space, we have $\mathcal{M} (\mathcal{X})$ to be a compact convex space, which, in turn implies $\Phi$ has a fixed point. In other words, there exists a probability measure $m \in \mathcal{M} (\mathcal{X})$ such that $\Phi (m) = m$, {\it i.e.}, $\mathcal{L}_{f}^{*} (m) = \Lambda m$, where $\displaystyle{\Lambda = \int_{\mathcal{X}} \mathcal{L}_{f} (1) \mathrm{d}m}$. 
\end{proof} 

\begin{claim} 
\label{claimtwo} 
There exists some $h \in \mathcal{C} \left( \mathcal{X}, \mathbb{R} \right)$ such that $\mathcal{L}_{f} h = \Lambda h$. 
\end{claim} 

\begin{proof} 
Let $x, y \in \mathcal{X}$. Define 
\begin{eqnarray} 
\label{sxy} 
\mathcal{S} (x, y) & = & \sup_{\nu\, \ge\, 1}\ \ \sup_{\boldsymbol{\beta}\, \in\, {\rm Cyl}_{\nu}}\ \ \Big\{ \mathcal{Z}_{\nu}^{-} (f) \left( \mathfrak{X}_{\nu}^{-} \left( x; \boldsymbol{\beta} \right) \right)\ -\ \mathcal{Z}_{\nu}^{-} (f) \left( \mathfrak{X}_{\nu}^{-} \left( y; \boldsymbol{\beta} \right) \right) \nonumber \\ 
& & \hspace{+0.5cm}  :\ d_{\overline{\mathbb{C}}} \left( \Pi_{j}^{-} \left( \mathfrak{X}_{\nu}^{-} \left( x; \boldsymbol{\beta} \right) \right), \Pi_{j}^{-} \left( \mathfrak{X}_{\nu}^{-} \left( y; \boldsymbol{\beta} \right) \right) \right) < \frac{d_{\overline{\mathbb{C}}} \left(x, y\right)}{\lambda^{j}}\ \forall 1 \le j \le \nu \Big\}. 
\end{eqnarray} 
Suppose the points $x, y \in \mathcal{X}$ are chosen such that $d_{\overline{\mathbb{C}}} \left(x, y\right) \le \dfrac{1}{\lambda^{\nu - 1}}$, then 
\[ \mathcal{S}(x, y)\ \ \le\ \ \sum\limits_{k\, \ge\, \nu + 1} \omega_{k} (f)\  \to\ 0,\ \ \text{as}\ \nu\ \text{grows larger}. \] 
We now define a subset of $\mathcal{C} \left( \mathcal{X}, \mathbb{R} \right)$ denoted by $\Omega \left( \mathcal{X}, \mathbb{R} \right)$ given by 
\begin{eqnarray*} 
\Omega \left( \mathcal{X}, \mathbb{R} \right) & = & \Big\{ g \in \mathcal{C} \left( \mathcal{X}, \mathbb{R} \right)\ :\ g \ge 0,\ \int_{\mathcal{X}} g \mathrm{d}\mu = 1\ \text{and for any} \\ 
& & \hspace{+5cm} \text{pair of points}\ x, y \in \mathcal{X},\ g(x) \le e^{\mathcal{S} (x, y)} g(y) \Big\}. 
\end{eqnarray*} 
Observe that $\Omega \left( \mathcal{X}, \mathbb{R} \right)$ is a convex subspace of $\mathcal{C} \left( \mathcal{X}, \mathbb{R} \right)$. Moreover, given $\epsilon > 0$, we can find $\nu \in \mathbb{Z}_{+}$ such that for all points $x, y \in \mathcal{X}$ with $d(x, y) \le \dfrac{1}{\lambda^{\nu - 1}}$, we have 
\[ |g(x) - g(y)| \le \left| e^{\max{ \left\{ \mathcal{S} (x, y),\; \mathcal{S} (y, x) \right\} }} - 1 \right| \le \epsilon. \] 
In other words, $\Omega \left( \mathcal{X}, \mathbb{R} \right)$ is an equicontinuous family. Further, $\Omega \left( \mathcal{X}, \mathbb{R} \right)$ is also a bounded family. 

For any $g \in \Omega \left( \mathcal{X}, \mathbb{R} \right)$, define an operator $\mathcal{T}$ by $\mathcal{T}(g) = \dfrac{\mathcal{L}_{f} g}{\Lambda}$. Then, by definition $\mathcal{T}(g) \ge 0$ and 
\[ \int \mathcal{T}(g) \mathrm{d}m\ \ =\ \ \int \frac{\mathcal{L}_{f} g}{\Lambda} \mathrm{d}m\ \ =\ \ \int g \mathrm{d}\left(\frac{\mathcal{L}_{f}^{*} m}{\Lambda}\right)\ \ =\ \ \int g \mathrm{d}m\ \ =\ \ 1. \] 
Further, for any $x \in \mathcal{X}$, we have 
\[ \mathcal{T}(g) (x)\ =\ \frac{1}{\Lambda} \mathcal{L}_{f} g (x)\ =\ \frac{1}{\Lambda} \sum_{\boldsymbol{\beta}\, \in\, {\rm Cyl}_{1}}\ \sum_{\mathfrak{X}_{1}^{-} \left( x; \boldsymbol{\beta} \right)\, \in\, \mathscr{Q}_{1}^{\Gamma} (x)} e^{\mathcal{Z}_{1}^{-} (f) \left( \mathfrak{X}_{1}^{-} \left( x; \boldsymbol{\beta} \right) \right)} g \left( \Pi_{1}^{-} \left( \mathfrak{X}_{1}^{-} \left( x; \boldsymbol{\beta} \right) \right) \right). \] 
To every point $\mathfrak{X}_{1}^{-} \left( x; \boldsymbol{\beta} \right) \in \mathscr{Q}_{1}^{\Gamma} (x)$, we associate a unique point $\mathfrak{X}_{1}^{-} \left( y; \boldsymbol{\beta} \right) \in \mathscr{Q}_{1}^{\Gamma} (y)$ for some $y \in \mathcal{X}$, in such a manner that $\lambda d_{\overline{\mathbb{C}}} \left( \Pi_{1}^{-} \left( \mathfrak{X}_{1}^{-} \left( x; \boldsymbol{\beta} \right) \right), \Pi_{1}^{-} \left( \mathfrak{X}_{1}^{-} \left( y; \boldsymbol{\beta} \right) \right) \right) < d_{\overline{\mathbb{C}}} (x, y)$ and the points $\mathfrak{X}_{1}^{-} \left( y; \boldsymbol{\beta} \right)$ are not repeated. This is possible due to Definition \eqref{expcorr}. Then, for such associated points $\mathfrak{X}_{1}^{-} \left( x; \boldsymbol{\beta} \right)$ and $\mathfrak{X}_{1}^{-} \left( y; \boldsymbol{\beta} \right)$, we have 
\begin{eqnarray*} 
& & e^{\mathcal{Z}_{1}^{-} (f) \left( \mathfrak{X}_{1}^{-} \left( x; \boldsymbol{\beta} \right) \right)} g \left( \Pi_{1}^{-} \left( \mathfrak{X}_{1}^{-} \left( x; \boldsymbol{\beta} \right) \right) \right) \\ 
\vspace{+10pt} \\ 
& \le & \frac{e^{\left[ \mathcal{Z}_{1}^{-} (f) \left( \mathfrak{X}_{1}^{-} \left( y; \boldsymbol{\beta} \right) \right)\, +\, \mathcal{Z}_{1}^{-} (f) \left( \mathfrak{X}_{1}^{-} \left( x; \boldsymbol{\beta} \right) \right)\, +\, \mathcal{S} \left( \Pi_{1}^{-} \left( \mathfrak{X}_{1}^{-} \left( x; \boldsymbol{\beta} \right) \right),\, \Pi_{1}^{-} \left( \mathfrak{X}_{1}^{-} \left( y; \boldsymbol{\beta} \right) \right) \right) \right]}}{e^{\mathcal{Z}_{1}^{-} (f) \left( \mathfrak{X}_{1}^{-} \left( y; \boldsymbol{\beta} \right) \right)}} \times\; g \left( \Pi_{1}^{-} \left( \mathfrak{X}_{1}^{-} \left( y; \boldsymbol{\beta} \right) \right) \right) \\ 
\vspace{+10pt} \\ 
& \le & e^{\left[ \mathcal{Z}_{1}^{-} (f) \left( \mathfrak{X}_{1}^{-} \left( y; \boldsymbol{\beta} \right) \right)\, +\, \mathcal{S} (x, y) \right]} g \left( \Pi_{1}^{-} \left( \mathfrak{X}_{1}^{-} \left( y; \boldsymbol{\beta} \right) \right) \right). 
\end{eqnarray*} 
Thus, 
\[ \mathcal{T}(g) (x)\ \le\ \frac{1}{\Lambda}\; e^{\mathcal{S} (x, y)}\; \mathcal{L}_{f} g (y)\ =\ e^{\mathcal{S} (x, y)}\; \left( \mathcal{T} g \right) (y). \] 
Hence, for every $g \in \Omega \left( \mathcal{X}, \mathbb{R} \right)$, we have $\mathcal{T}g \in \Omega \left( \mathcal{X}, \mathbb{R} \right)$. One can then apply the Schauder-Tychonov theorem to the operator $\mathcal{T}$ defined on $\Omega \left( \mathcal{X}, \mathbb{R} \right)$ to conclude the existence of a fixed point of $\mathcal{T}$ in $\Omega \left( \mathcal{X}, \mathbb{R} \right)$; {\it i.e.}, there exists a function, say $h \in \Omega \left( \mathcal{X}, \mathbb{R} \right)$ that satisfies $\mathcal{T} h = h$, meaning $h$ is an eigenfunction of $\mathcal{L}_{f}$ with eigenvalue $\Lambda;\ \mathcal{L}_{f} h = \Lambda h$. 
\end{proof} 

\begin{claim} 
$\Lambda$ is a simple eigenvalue for $\mathcal{L}_{f}$. 
\end{claim} 

\begin{proof} 
Let $g$ and $h$ be any two distinct eigenfunctions of the operator $\mathcal{L}_{f}$ corresponding to the eigenvalue $\Lambda$. Suppose there exists a point $x \in \mathcal{X}$ such that $g(x) > 0$. Then let $s = \sup \left\{t \ge 0 : h - t g \ge 0 \right\}$. Since the limit $s$ exists, we know that there exists a point $y \in \mathcal{X}$ such that $(h - sg)(y) = 0$. Hence, for all $\nu \in \mathbb{Z}^{+}$, we have 
\begin{eqnarray*} 
0 & = & \Lambda^{\nu} \left( h - sg \right) (y)\ \ =\ \ \mathcal{L}_{f}^{\nu} \left( h - sg \right) (y) \\ 
\vspace{+10pt} \\ 
& = & \sum_{\boldsymbol{\beta}\, \in\, {\rm Cyl}_{\nu}}\ \ \sum_{\mathfrak{X}_{\nu}^{-} \left( y; \boldsymbol{\beta} \right)\, \in\, \mathscr{Q}_{\nu}^{\Gamma} (y)} e^{\mathcal{Z}_{\nu}^{-} (f) \left( \mathfrak{X}_{\nu}^{-} \left( y; \boldsymbol{\beta} \right) \right)} \left( h - sg \right) \left( \Pi_{\nu}^{-} \left( \mathfrak{X}_{\nu}^{-} \left( y; \boldsymbol{\beta} \right) \right) \right). 
\end{eqnarray*} 
Since each term in the summand is non-negative, we have 
\[ \left( h - sg \right) \left( \Pi_{\nu}^{-} \left( \mathfrak{X}_{\nu}^{-} \left( y; \boldsymbol{\beta} \right) \right) \right)\ \ =\ \ 0\ \ \ \text{for every}\ \boldsymbol{\beta} \in {\rm Cyl}_{\nu}\ \ \text{and for all}\ \nu \in \mathbb{Z}^{+}. \] 
However, since $\displaystyle{\bigcup\limits_{\nu\, \in\, \mathbb{Z}_{+}} \bigcup\limits_{\boldsymbol{\beta}\, \in\, {\rm Cyl}_{\nu}} \left\{ \Pi_{\nu}^{-} \left( \mathfrak{X}_{\nu}^{-} \left( y; \boldsymbol{\beta} \right) \right) \right\}}$ is a dense subset of $\mathcal{X}$ for any $y \in \mathcal{X}$, we conclude that $h = sg$ on $\mathcal{X}$. Thus, $\Lambda$ is a simple eigenvalue of $\mathcal{L}_{f}$ with $h$ as its corresponding eigenfunction. 
\end{proof} 

Hence, the proof of Theorem \eqref{rot} is complete, provided we prove the following claim. We postpone the same to Section \eqref{proof}. 

%\begin{claim} 
%\label{maximal}
%$\Lambda = e^{{\rm Pr} (\Gamma, f)}$. 
%\end{claim} 

\begin{claim} 
\label{ucgce}
$\left\{ \dfrac{1}{\Lambda^{\nu}} \mathcal{L}_{f}^{\nu} (g) \right\}_{\nu\, \ge\, 1}$ converges uniformly in $\mathcal{X}$ for every $g \in \mathcal{C} \left( \mathcal{X}, \mathbb{R} \right)$. 
\end{claim} 

\section{The classical Ruelle operator} 

We start this section with the statement of a theorem due to Londhe as in \cite{ml:2022} that proves a certain minimal recurrence in the forward orbit for every point in the support of the Dinh-Sibony measure. 

\begin{theorem} \cite{ml:2022} 
\label{mlt} 
Let $M$ be a meromorphic correspondence of topological degree $d$ on a compact
complex manifold $X$. Let $\mu$ be a $M^{*}$-invariant Borel probability measure defined on $X$ such that it does not put any mass on pluripolar sets. Then, for any point $x \in {\rm supp}(\mu)$, we have $M(x) \cap {\rm supp}(\mu) \neq \emptyset$, where $M(x)$ denotes the collection of all $y \in X$ such that $(x,y) \in {\rm graph}(M)$. 
\end{theorem}

Now using Theorem \eqref{mlt}, we define a subset of $\mathscr{P}^{\Gamma} \left( \mathcal{X} \right)$ wherein we collect the forward images of a point $x_{0} \in \mathcal{X}$ that also remains in $\mathcal{X}$, {\it i.e.}, define 
\[ \mathscr{P}^{\mu, \Gamma} \left( \mathcal{X} \right)\ :=\ \left\{ \left( x_{0}, x_{1}, x_{2}, \cdots ; \alpha_{1}, \alpha_{2}, \cdots \right) \in \mathscr{P}^{\Gamma} \left( \mathcal{X} \right) : x_{\nu} \in \mathcal{X}\  \forall \nu \in \mathbb{Z}_{+} \right\}.\]  
Note that, by definition any $\left(x_{i - 1}, x_{i}\right) \in \Gamma\vert_{\mathcal{X} \times \mathcal{X}}$ for all $i \ge 1$. Further, $\mathscr{P}^{\mu, \Gamma} \left( \mathcal{X} \right)$ is an automatic metric space with respect to the appropriate restriction of the metric $d_{\mathscr{P}^{\Gamma} \left( \overline{\mathbb{C}} \right)}$, defined in Equation \eqref{metriconinfspace}. Further, $\mathscr{P}^{\mu, \Gamma} \left( \mathcal{X} \right)$ is a closed subset of $\mathscr{P}^{\Gamma} \left( \widehat{\mathbb{C}} \right)$, and hence, compact. Also note that the restriction of the shift map $\sigma$, as defined in Equation \eqref{shiftmap}, on $\mathscr{P}^{\mu, \Gamma} \left( \mathcal{X} \right)$ yields the map to be forward invariant. Since we know $\mathcal{X}$ to be backward invariant under the action of the correspondence $\Gamma$, we essentially have $\sigma$ to be completely invariant on $\mathscr{P}^{\mu, \Gamma} \left( \mathcal{X} \right)$ {\it i.e.}, $\sigma \left(\mathscr{P}^{\mu, \Gamma} \left( \mathcal{X} \right) \right) = \mathscr{P}^{\mu, \Gamma} \left( \mathcal{X} \right)$. 

As considered in the proof of Theorem \eqref{variational pre}, we now observe that for every $f \in \mathcal{C} \left( \mathcal{X}, \mathbb{R} \right)$, there exists a function $F = f \circ \Pi_{0}^{+} \in  \mathcal{C} \left( \mathscr{P}^{\mu, \Gamma} \left( \mathcal{X} \right), \mathbb{R} \right)$. Let $\mathcal{L}_{F}$ denote the classical Ruelle operator on $\mathcal{C} \left( \mathscr{P}^{\mu, \Gamma} \left( \mathcal{X} \right), \mathbb{R} \right)$ where the dynamics in the underlying space happens through the shift map $\sigma$. Then, for any $G \in \mathcal{C} \left( \mathscr{P}^{\mu, \Gamma} \left( \mathcal{X} \right), \mathbb{R} \right)$, the action of the classical Ruelle operator is given by 
\[ \mathcal{L}_{F} (G) \left( \mathfrak{X}^{+} \left( x_{0}; \boldsymbol{\alpha} \right) \right)\ \ =\ \ \sum_{\mathfrak{X}^{+} \left( x^{*}; \boldsymbol{\alpha}^{*} \right)\ \in\ \sigma^{-1} \left( \mathfrak{X}^{+} \left( x_{0}; \boldsymbol{\alpha} \right) \right)} \hspace{-0.5cm} e^{F \left( \mathfrak{X}^{+} \left( x^{*}; \boldsymbol{\alpha}^{*} \right) \right)} G \left( \mathfrak{X}^{+} \left( x^{*}; \boldsymbol{\alpha}^{*} \right) \right), \] 
so that for any $\nu \in \mathbb{Z}_{+}$, we have 
\begin{eqnarray*} 
\left( \mathcal{L}_{F} \right)^{\nu} (G) \left( \mathfrak{X}^{+} \left( x_{0}; \boldsymbol{\alpha} \right) \right) & = & \sum_{\boldsymbol{\gamma}\, \in\, {\rm Cyl}_{\nu}}\ \sum_{\mathfrak{X}^{+} \left( x_{- \nu}; \boldsymbol{\gamma \alpha} \right)\ \in\ \sigma^{- \nu} \left( \mathfrak{X}^{+} \left( x_{0}; \boldsymbol{\alpha} \right) \right)} \hspace{-1cm} e^{ \sum_{j\, =\, 0}^{\nu - 1} F \left( \sigma^{j} \left( \mathfrak{X}^{+} \left( x_{- \nu}; \boldsymbol{\gamma \alpha} \right) \right) \right)} G \left( \mathfrak{X}^{+} \left( x_{- \nu}; \boldsymbol{\gamma \alpha} \right) \right) \\ 
& = & \sum_{\boldsymbol{\gamma}\, \in\, {\rm Cyl}_{\nu}}\ \sum_{\mathfrak{X}^{+} \left( x_{- \nu}; \boldsymbol{\gamma \alpha} \right)\ \in\ \sigma^{- \nu} \left( \mathfrak{X}^{+} \left( x_{0}; \boldsymbol{\alpha} \right) \right)} \hspace{-1cm} e^{ F_{\nu} \left( \mathfrak{X}^{+} \left( x_{- \nu}; \boldsymbol{\gamma \alpha} \right) \right)} G \left( \mathfrak{X}^{+} \left( x_{- \nu}; \boldsymbol{\gamma \alpha} \right) \right). 
\end{eqnarray*} 
Here, the notation in the last equation $F_{\nu} ( \cdot )$ is as defined in Equation \eqref{ergodsum} and $\boldsymbol{\gamma \alpha}$ merely represents the concatenation of the $\nu$-long word $\boldsymbol{\gamma}$ to the infinitely long word $\boldsymbol{\alpha}$, in the manner written. 

We now establish a relationship between the eigenvalue $\Lambda$ and the corresponding eigenfunction $h \in \mathcal{C} \left( \mathcal{X}, \mathbb{R} \right)$ of $\mathcal{L}_{f}$, as stated in Theorem \eqref{rot} and the respective quantities of the classical Ruelle operator $\mathcal{L}_{F}$. 

\begin{proposition} 
$\Lambda$ is an eigenvalue for $\mathcal{L}_{F}$ with corresponding eigenfunction $H = h \circ \Pi_{0}^{+}$, where $\Lambda$ and $h \in \mathcal{C} \left( \mathcal{X}, \mathbb{R} \right)$, are as stated in Theorem \eqref{rot}. 
\end{proposition} 

\begin{proof} 
For any point $\mathfrak{X}^{+} \left( x_{0}; \boldsymbol{\alpha} \right) \in \mathscr{P}^{\mu, \Gamma} \left( \mathcal{X} \right)$, consider the set $\sigma^{-1} \left( \mathfrak{X}^{+} \left( x_{0}; \boldsymbol{\alpha} \right) \right)$. Since $\mathcal{X}$ is backward invariant, there exists holomorphic branches of iteration of the correspondence such that $\mathfrak{X}^{+} \left( x^{*}; \boldsymbol{\alpha}^{*} \right)$ belongs to $\sigma^{-1} \left( \mathfrak{X}^{+} \left( x_{0}; \boldsymbol{\alpha} \right) \right) \subseteq \mathscr{P}^{\mu, \Gamma} \left( \mathcal{X} \right)$. 

Let the set $\left\{ y_{1}, y_{2}, \cdots, y_{d_{1}} \right\}$ exhaust the collection of all $y_{j}$'s in $\mathcal{X}$ such that $(y_{j}, x_{0}) \in \Gamma$. Here, $d_{1}$ is as defined in Equation \eqref{done}. Then there exists $\gamma_{j} \in \left\{ 1, 2, \cdots, N \right\}$ such that $(y_{j}, x_{0}) \in \Gamma_{\gamma_{j}}$. Concatenating $\gamma_{j}$ with $\boldsymbol{\alpha}$, we observe that $\mathfrak{X}^{+} \left( y_{1}; \gamma_{1} \boldsymbol{\alpha} \right), \mathfrak{X}^{+} \left( y_{2}; \gamma_{2} \boldsymbol{\alpha} \right), \cdots, \mathfrak{X}^{+} \left( y_{d_{1}}; \gamma_{d_{1}} \boldsymbol{\alpha} \right)$ are the only preimages of $\mathfrak{X}^{+} \left( x_{0}; \boldsymbol{\alpha} \right)$ under $\sigma$. Hence, 
\begin{eqnarray*} 
\mathcal{L}_{F} H \left( \mathfrak{X}^{+} \left( x_{0}; \boldsymbol{\alpha} \right) \right) & = & \sum_{1\, \le\, j\, \le\, d_{1}} e^{F \left( \mathfrak{X}^{+} \left( y_{j}; \gamma_{j} \boldsymbol{\alpha} \right) \right)} H \left( \mathfrak{X}^{+} \left( y_{j}; \gamma_{j} \boldsymbol{\alpha} \right) \right) \\ 
& = & \sum_{1\, \le\, j\, \le\, d_{1}} e^{f(y_{j})} h(y_{j})\ =\ \mathcal{L}_{f} h (x_{0})\ =\ \Lambda h(x_{0}) \\ 
& = & \Lambda H \left( \mathfrak{X}^{+} \left( x_{0}; \boldsymbol{\alpha} \right) \right). 
\end{eqnarray*} 
\end{proof}  

Normalising the classical Ruelle operator $\mathcal{L}_{F}$ along the lines of Haydn, as in \cite{H:1999}, we consider 
\begin{equation} 
\label{normalF} 
F^{*}\ \ =\ \ F - \log \left( H \circ \sigma \right) + \log H - \log \Lambda \in \mathcal{C} \left( \mathscr{P}^{\mu, \Gamma} \left( \mathcal{X} \right), \mathbb{R} \right), 
\end{equation}  
in order that we obtain 
\begin{eqnarray} 
\label{normalisedruelle}
\left( \mathcal{L}_{F^{*}} 1 \right) \left( \mathfrak{X}^{+} \left( x_{0}; \boldsymbol{\alpha} \right) \right) & = & \sum_{\mathfrak{X}^{+} \left( x^{*}; \boldsymbol{\alpha}^{*} \right)\ \in\ \sigma^{-1} \left( \mathfrak{X}^{+} \left( x_{0}; \boldsymbol{\alpha} \right) \right)} \hspace{-1cm} e^{ \left( F - \log \left( H \circ \sigma \right) + \log H - \log \Lambda \right) \left( \mathfrak{X}^{+} \left( x^{*}; \boldsymbol{\alpha}^{*} \right) \right)} 1 \left( \mathfrak{X}^{+} \left( x^{*}; \boldsymbol{\alpha}^{*} \right) \right) \nonumber \\ 
\vspace{+10pt} \nonumber \\ 
& = & \frac{\left( \mathcal{L}_{F} H \right) \left( \mathfrak{X}^{+} \left( x_{0}; \boldsymbol{\alpha} \right) \right)}{\Lambda H \left( \mathfrak{X}^{+} \left( x_{0}; \boldsymbol{\alpha} \right) \right)}\ \ =\ \ 1. 
\end{eqnarray} 

\section{Equicontinuity of $\left\{ \mathcal{L}_{F^{*}}^{\nu} \left( \cdot \right) \right\}$} 

Since the holomorphic correspondence $\Gamma$ is expansive and satisfies the conditions in Definition \eqref{expcorr}, it is easy to verify that whenever we take two points $\mathfrak{X}^{+} \left( x_{0}; \boldsymbol{\alpha} \right)$ and $\mathfrak{X}^{+} \left( y_{0}; \boldsymbol{\beta} \right)$ in $\mathscr{P}^{\mu, \Gamma} \left( \mathcal{X} \right)$, there exist $\mathfrak{X}^{+} \left( x_{-1}; \gamma \boldsymbol{\alpha} \right) \in \sigma^{-1} \left( \mathfrak{X}^{+} \left( x_{0}; \boldsymbol{\alpha} \right) \right)$ and $\mathfrak{X}^{+} \left( y_{-1}; \gamma \boldsymbol{\beta} \right) \in \sigma^{-1} \left( \mathfrak{X}^{+} \left( y_{0}; \boldsymbol{\beta} \right) \right)$ for some $\gamma \in \left\{ 1, 2, \cdots, N \right\}$ such that 
\begin{equation} 
\label{inflamb} 
d_{\mathscr{P}^{\mu, \Gamma} \left( \mathcal{X} \right)} \left( \mathfrak{X}^{+} \left( x_{-1}; \gamma \boldsymbol{\alpha} \right), \mathfrak{X}^{+} \left( y_{-1}; \gamma \boldsymbol{\beta} \right) \right)\ \ <\ \ \frac{1}{\lambda}\ d_{\mathscr{P}^{\mu, \Gamma} \left( \mathcal{X} \right)} \left( \mathfrak{X}^{+} \left( x_{0}; \boldsymbol{\alpha} \right), \mathfrak{X}^{+} \left( y_{0}; \boldsymbol{\alpha} \right) \right), 
\end{equation} 
where $\lambda > 1$ is the expansivity constant of the correspondence, as defined in condition (2) of Definition \eqref{expcorr}. 

We now define a quantity called $\mathcal{S} \left( \mathfrak{X}^{+} \left( x_{0}; \boldsymbol{\alpha} \right), \mathfrak{X}^{+} \left( y_{0}; \boldsymbol{\beta} \right) \right)$ analogous to the one defined in Section \eqref{expcorrsec}, in Equation \eqref{sxy}. Here, $\mathfrak{X}^{+} \left( x_{0}; \boldsymbol{\alpha} \right)$ and $\mathfrak{X}^{+} \left( y_{0}; \boldsymbol{\beta} \right)$ are arbitrary points in $\mathscr{P}^{\mu, \Gamma} \left( \mathcal{X} \right)$ while $F^{*} \in \mathcal{C} \left( \mathscr{P}^{\mu, \Gamma} \left( \mathcal{X} \right), \mathbb{R} \right)$ is normalised, as described above. 
\begin{eqnarray*} 
\mathcal{S} \left( \mathfrak{X}^{+} \left( x_{0}; \boldsymbol{\alpha} \right), \mathfrak{X}^{+} \left( y_{0}; \boldsymbol{\beta} \right) \right) & = & \sup_{\nu\, \ge\, 1}\ \ \sup_{\boldsymbol{\gamma}\, \in\, {\rm Cyl}_{\nu}}\ \ \Big\{ F^{*}_{\nu} \left( \mathfrak{X}^{+} \left( x_{- \nu}; \boldsymbol{\gamma \alpha} \right) \right)\ -\ F^{*}_{\nu} \left( \mathfrak{X}^{+} \left( y_{- \nu}; \boldsymbol{\gamma \beta} \right) \right) \nonumber \\ 
& : & \mathfrak{X}^{+} \left( x_{- \nu}; \boldsymbol{\gamma \alpha} \right) \in \sigma^{- \nu} \left( \mathfrak{X}^{+} \left( x_{0}; \boldsymbol{\alpha} \right) \right), \\ 
& & \mathfrak{X}^{+} \left( y_{- \nu}; \boldsymbol{\gamma \beta} \right) \in \sigma^{- \nu} \left( \mathfrak{X}^{+} \left( y_{0}; \boldsymbol{\beta} \right) \right)\ \text{satisfying} \\ 
\frac{d_{\overline{\mathbb{C}}} \left(x, y\right)}{\lambda^{j}} & > & d_{\overline{\mathbb{C}}} \left( \Pi_{j}^{+} \left( \mathfrak{X}^{+} \left( x_{- \nu}; \boldsymbol{\gamma \alpha} \right) \right), \Pi_{j}^{+} \left( \mathfrak{X}^{+} \left( y_{- \nu}; \boldsymbol{\gamma \beta} \right) \right) \right)\ \forall 1 \le j \le \nu \Big\}. 
\end{eqnarray*} 

We now prove that $\left\{ \mathcal{L}_{F^{*}}^{\nu} \left( G \right) \right\}_{\nu\, \ge\, 1}$ is an equicontinuous family for any $G \in \mathcal{C} \left( \mathscr{P}^{\mu, \Gamma} \left( \mathcal{X} \right), \mathbb{R} \right)$. 

\begin{theorem} 
$\left\{ \mathcal{L}_{F^{*}}^{\nu} \left( G \right) \right\}_{\nu\, \ge\, 1}$ is an equicontinuous family for any $G \in \mathcal{C} \left( \mathscr{P}^{\mu, \Gamma} \left( \mathcal{X} \right), \mathbb{R} \right)$. 
\end{theorem} 

\begin{proof}
Let $\mathfrak{X}^{+} \left( x_{0}; \boldsymbol{\alpha} \right)$ and $\mathfrak{X}^{+} \left( y_{0}; \boldsymbol{\beta} \right)$ be any two arbitrary points in $\mathscr{P}^{\mu, \Gamma} \left( \mathcal{X} \right)$ such that 
\begin{equation} 
\label{dist} 
d_{\mathscr{P}^{\mu, \Gamma} \left( \mathcal{X} \right)} \left( \mathfrak{X}^{+} \left( x_{0}; \boldsymbol{\alpha} \right), \mathfrak{X}^{+} \left( y_{0}; \boldsymbol{\beta} \right) \right)\ \ \le\ \ \frac{1}{\lambda^{k - 1}}\ \text{for some fixed}\ k \in \mathbb{Z}_{+}. 
\end{equation} 

To every pair of points $\mathfrak{X}^{+} \left( x_{0}; \boldsymbol{\alpha} \right)$ and $\mathfrak{X}^{+} \left( y_{0}; \boldsymbol{\beta} \right)$ in $\mathscr{P}^{\mu, \Gamma} \left( \mathcal{X} \right)$, we associate a unique pair of points given by $\mathfrak{X}^{+} \left( x_{- j}; \boldsymbol{\gamma \alpha} \right) \in \sigma^{- j} \left( \mathfrak{X}^{+} \left( x_{0}; \boldsymbol{\alpha} \right) \right)$ and $\mathfrak{X}^{+} \left( y_{- j}; \boldsymbol{\gamma \beta} \right) \in \sigma^{- j} \left( \mathfrak{X}^{+} \left( y_{0}; \boldsymbol{\beta} \right) \right)$ for some $\boldsymbol{\gamma} \in {\rm Cyl}_{j}$ that satisfies  
\[ d_{\mathscr{P}^{\mu, \Gamma} \left( \mathcal{X} \right)} \left( \mathfrak{X}^{+} \left( x_{-j}; \boldsymbol{\gamma \alpha} \right), \mathfrak{X}^{+} \left( y_{-j}; \boldsymbol{\gamma \beta} \right) \right)\ \ <\ \ \frac{1}{\lambda^{j}}\ d_{\mathscr{P}^{\mu, \Gamma} \left( \mathcal{X} \right)} \left( \mathfrak{X}^{+} \left( x_{0}; \boldsymbol{\alpha} \right), \mathfrak{X}^{+} \left( y_{0}; \boldsymbol{\alpha} \right) \right), \] 
for every $1 \le j \le \nu$. Observe that such an association is possible due to Equation \eqref{inflamb}. Then, for such associated pairs of points, consider 
\begin{eqnarray*} 
& & \left| \mathcal{L}_{F^{*}}^{\nu} (G) \left( \mathfrak{X}^{+} \left( x_{0}; \boldsymbol{\alpha} \right) \right) - \mathcal{L}_{F^{*}}^{\nu} (G) \left( \mathfrak{X}^{+} \left( y_{0}; \boldsymbol{\beta} \right) \right) \right| \\ 
\vspace{+10pt} \\ 
& = & \left| \sum_{\boldsymbol{\gamma}\, \in\, {\rm Cyl}_{\nu}}\ \sum_{\mathfrak{X}^{+} \left( x_{- \nu}; \boldsymbol{\gamma \alpha} \right)\ \in\ \sigma^{- \nu} \left( \mathfrak{X}^{+} \left( x_{0}; \boldsymbol{\alpha} \right) \right)} \hspace{-1cm} e^{ F^{*}_{\nu} \left( \mathfrak{X}^{+} \left( x_{- \nu}; \boldsymbol{\gamma \alpha} \right) \right)} G \left( \mathfrak{X}^{+} \left( x_{- \nu}; \boldsymbol{\gamma \alpha} \right) \right) \right. \\ 
& & \left. \hspace{+3cm} -\ \sum_{\boldsymbol{\gamma}\, \in\, {\rm Cyl}_{\nu}}\ \sum_{\mathfrak{X}^{+} \left( y_{- \nu}; \boldsymbol{\gamma \beta} \right)\ \in\ \sigma^{- \nu} \left( \mathfrak{X}^{+} \left( y_{0}; \boldsymbol{\beta} \right) \right)} \hspace{-1cm} e^{ F^{*}_{\nu} \left( \mathfrak{X}^{+} \left( y_{- \nu}; \boldsymbol{\gamma \beta} \right) \right)} G \left( \mathfrak{X}^{+} \left( y_{- \nu}; \boldsymbol{\gamma \beta} \right) \right) \right| \\ 
\vspace{+10pt} \\ 
& \le & \left| \sum_{\boldsymbol{\gamma}\, \in\, {\rm Cyl}_{\nu}} \sum_{\substack{\mathfrak{X}^{+} \left( x_{- \nu}; \boldsymbol{\gamma \alpha} \right)\; \in\; \sigma^{- \nu} \left( \mathfrak{X}^{+} \left( x_{0}; \boldsymbol{\alpha} \right) \right) \\ \mathfrak{X}^{+} \left( y_{- \nu}; \boldsymbol{\gamma \beta} \right)\; \in\; \sigma^{- \nu} \left( \mathfrak{X}^{+} \left( y_{0}; \boldsymbol{\beta} \right) \right)}} \hspace{-1cm} \left[ e^{ F^{*}_{\nu} \left( \mathfrak{X}^{+} \left( x_{- \nu}; \boldsymbol{\gamma \alpha} \right) \right)} - e^{ F^{*}_{\nu} \left( \mathfrak{X}^{+} \left( y_{- \nu}; \boldsymbol{\gamma \beta} \right) \right)} \right] G \left( \mathfrak{X}^{+} \left( x_{- \nu}; \boldsymbol{\gamma \alpha} \right) \right) \right| \\ 
& + & \left| \sum_{\boldsymbol{\gamma}\, \in\, {\rm Cyl}_{\nu}} \sum_{\substack{\mathfrak{X}^{+} \left( x_{- \nu}; \boldsymbol{\gamma \alpha} \right)\; \in\; \sigma^{- \nu} \left( \mathfrak{X}^{+} \left( x_{0}; \boldsymbol{\alpha} \right) \right) \\ \mathfrak{X}^{+} \left( y_{- \nu}; \boldsymbol{\gamma \beta} \right)\; \in\; \sigma^{- \nu} \left( \mathfrak{X}^{+} \left( y_{0}; \boldsymbol{\beta} \right) \right)}} \hspace{-1cm} e^{ F^{*}_{\nu} \left( \mathfrak{X}^{+} \left( y_{- \nu}; \boldsymbol{\gamma \beta} \right) \right)} \left[ G \left( \mathfrak{X}^{+} \left( x_{- \nu}; \boldsymbol{\gamma \alpha} \right) \right) - G \left( \mathfrak{X}^{+} \left( y_{- \nu}; \boldsymbol{\gamma \beta} \right) \right) \right] \right| \\ 
\vspace{+10pt} \\ 
& \le & \left\| G \right\| \left| \sum_{\boldsymbol{\gamma}\, \in\, {\rm Cyl}_{\nu}} \sum_{\substack{\mathfrak{X}^{+} \left( x_{- \nu}; \boldsymbol{\gamma \alpha} \right)\; \in\; \sigma^{- \nu} \left( \mathfrak{X}^{+} \left( x_{0}; \boldsymbol{\alpha} \right) \right) \\ \mathfrak{X}^{+} \left( y_{- \nu}; \boldsymbol{\gamma \beta} \right)\; \in\; \sigma^{- \nu} \left( \mathfrak{X}^{+} \left( y_{0}; \boldsymbol{\beta} \right) \right)}} \hspace{-1cm} \left[ e^{ F^{*}_{\nu} \left( \mathfrak{X}^{+} \left( x_{- \nu}; \boldsymbol{\gamma \alpha} \right) \right)} - e^{ F^{*}_{\nu} \left( \mathfrak{X}^{+} \left( y_{- \nu}; \boldsymbol{\gamma \beta} \right) \right)} \right] \right| \\ 
& + & \omega_{k + \nu} \left(G\right) \left| \sum_{\boldsymbol{\gamma}\, \in\, {\rm Cyl}_{\nu}} \sum_{\mathfrak{X}^{+} \left( y_{- \nu}; \boldsymbol{\gamma \beta} \right)\; \in\; \sigma^{- \nu} \left( \mathfrak{X}^{+} \left( y_{0}; \boldsymbol{\beta} \right) \right)} \hspace{-1cm} e^{ F^{*}_{\nu} \left( \mathfrak{X}^{+} \left( y_{- \nu}; \boldsymbol{\gamma \beta} \right) \right)} \right|, 
\end{eqnarray*} 
where 
\[ \omega_{k} (G)\ :=\ \sup \left\{ \left| G \left( \mathfrak{X}^{+} \left( x_{0}; \boldsymbol{\alpha} \right) \right) - G \left( \mathfrak{X}^{+} \left( y_{0}; \boldsymbol{\beta} \right) \right) \right| : d_{\mathscr{P}^{\mu, \Gamma} \left( \mathcal{X} \right)} \left( \mathfrak{X}^{+} \left( x_{0}; \boldsymbol{\alpha} \right), \mathfrak{X}^{+} \left( y_{0}; \boldsymbol{\beta} \right) \right) \le \frac{1}{\lambda^{k - 1}} \right\}. \]  
Thus, 
\begin{eqnarray*} 
& & \left| \mathcal{L}_{F^{*}}^{\nu} (G) \left( \mathfrak{X}^{+} \left( x_{0}; \boldsymbol{\alpha} \right) \right) - \mathcal{L}_{F^{*}}^{\nu} (G) \left( \mathfrak{X}^{+} \left( y_{0}; \boldsymbol{\beta} \right) \right) \right| \\ 
\vspace{+10pt} \\ 
& \le & \left| \sum_{\boldsymbol{\gamma}\, \in\, {\rm Cyl}_{\nu}} \sum_{\mathfrak{X}^{+} \left( y_{- \nu}; \boldsymbol{\gamma \beta} \right)\; \in\; \sigma^{- \nu} \left( \mathfrak{X}^{+} \left( y_{0}; \boldsymbol{\beta} \right) \right)} \hspace{-1cm} e^{ F^{*}_{\nu} \left( \mathfrak{X}^{+} \left( y_{- \nu}; \boldsymbol{\gamma \beta} \right) \right)} \right| \left[ \left\| G \right\| \left| e^{\mathcal{S} \left( \mathfrak{X}^{+} \left( x_{0}; \boldsymbol{\alpha} \right), \mathfrak{X}^{+} \left( y_{0}; \boldsymbol{\beta} \right) \right)} - 1 \right|\ +\ \omega_{k + \nu} \left(G\right) \right] \\ 
& = & \left\| G \right\| \left| e^{\mathcal{S} \left( \mathfrak{X}^{+} \left( x_{0}; \boldsymbol{\alpha} \right), \mathfrak{X}^{+} \left( y_{0}; \boldsymbol{\beta} \right) \right)} - 1 \right|\ +\ \omega_{k} \left(G\right). 
\end{eqnarray*} 

Therefore, for any pair of points $\mathfrak{X}^{+} \left( x_{0}; \boldsymbol{\alpha} \right), \mathfrak{X}^{+} \left( y_{0}; \boldsymbol{\beta} \right) \in \mathscr{P}^{\mu, \Gamma} \left( \mathcal{X} \right)$, that satisfies Equation \eqref{dist}, we have $\mathcal{S} \left(\mathfrak{X}^{+} \left( x_{0}; \boldsymbol{\alpha} \right), \mathfrak{X}^{+} \left( y_{0}; \boldsymbol{\beta} \right) \right) \to 0$ as $\nu \to \infty$. Thus, the family $\left\{ \mathcal{L}_{F^{*}}^{\nu} \left( G \right) \right\}_{\nu\, \ge\, 1}$ is equicontinuous. 
\end{proof} 

\section{Density of preimages in $\mathscr{P}^{\mu, \Gamma} \left( \mathcal{X} \right)$} 

Since the family $\left\{ \mathcal{L}_{F^{*}}^{\nu} \left( G \right) \right\}$ is equicontinuous for any $G \in \mathcal{C} \left( \mathscr{P}^{\mu, \Gamma} \left( \mathcal{X} \right), \mathbb{R} \right)$, there exists a subsequence $\nu_{k} \in \mathbb{Z}_{+}$ such that $\left\{ \mathcal{L}_{F^{*}}^{\nu_{k}} \left( G \right) \right\}_{k\, \ge\, 1}$ is uniformly convergent. Say that the subsequence converges to the limit $G_{*} \in \mathcal{C} \left( \mathscr{P}^{\mu, \Gamma} \left( \mathcal{X} \right), \mathbb{R} \right)$. Since $\mathcal{L}_{F^{*}}^{\nu} \left( 1 \right) \equiv 1$ (from Equation \eqref{normalisedruelle}), we have 
\[ \sup G\ \ge\ \sup \mathcal{L}_{F^{*}} \left( G \right)\ \ge\ \cdots\ \ge\ \sup \mathcal{L}_{F^{*}}^{\nu} \left( G \right)\ \ge\ \cdots\ \ge\ \sup G_{*}, \] 
where all the suprema in the above inequality is taken over points in $\mathscr{P}^{\mu, \Gamma} \left( \mathcal{X} \right)$. Hence, as $k \in \mathbb{Z}_{+}$ grows beyond a threshold, we can find an $\epsilon_{k} > 0$ such that 
\[ G_{*} \left( \mathfrak{X}^{+} \left( x_{0}; \boldsymbol{\alpha} \right) \right) \ge \left( \mathcal{L}_{F^{*}}^{\nu_{k}} G \right) \left( \mathfrak{X}^{+} \left( x_{0}; \boldsymbol{\alpha} \right) \right) - \epsilon_{k}\ \  \forall \mathfrak{X}^{+} \left( x_{0}; \boldsymbol{\alpha} \right) \in \mathscr{P}^{\mu, \Gamma} \left( \mathcal{X} \right). \] 
Thus, for every $\nu \in \mathbb{Z}_{+}$, we have $\sup \mathcal{L}_{F^{*}}^{\nu} G_{*} \ge \sup G_{*}$ on $\mathscr{P}^{\mu, \Gamma} \left( \mathcal{X} \right)$. Since $\mathscr{P}^{\mu, \Gamma} \left( \mathcal{X} \right)$ is a compact metric space, there exist points, say $\mathfrak{X}^{+} \left( x_{0}; \boldsymbol{\alpha} \right)$ and $\mathfrak{X}^{+} \left( y_{0}; \boldsymbol{\beta} \right)$ in $\mathscr{P}^{\mu, \Gamma} \left( \mathcal{X} \right)$ such that $\sup G_{*} = G_{*} \left( \mathfrak{X}^{+} \left( x_{0}; \boldsymbol{\alpha} \right) \right)$ and $\sup \mathcal{L}_{F^{*}}^{\nu} G_{*} = \left( \mathcal{L}_{F^{*}}^{\nu} G_{*} \right) \left( \mathfrak{X}^{+} \left( y_{0}; \boldsymbol{\beta} \right) \right)$. Then, 
\begin{eqnarray} 
\label{usedensity} 
G_{*} \left( \mathfrak{X}^{+} \left( x_{0}; \boldsymbol{\alpha} \right) \right) & \le & \left( \mathcal{L}_{F^{*}}^{\nu} G_{*} \right) \left( \mathfrak{X}^{+} \left( y_{0}; \boldsymbol{\beta} \right) \right) \nonumber \\ 
& = & \sum_{\boldsymbol{\gamma}\, \in\, {\rm Cyl}_{\nu}}\ \sum_{\mathfrak{X}^{+} \left( y_{- \nu}; \boldsymbol{\gamma \beta} \right)\ \in\ \sigma^{- \nu} \left( \mathfrak{X}^{+} \left( y_{0}; \boldsymbol{\beta} \right) \right)} \hspace{-1cm} e^{ F^{*}_{\nu} \left( \mathfrak{X}^{+} \left( y_{- \nu}; \boldsymbol{\gamma \beta} \right) \right)} G_{*} \left( \mathfrak{X}^{+} \left( y_{- \nu}; \boldsymbol{\gamma \beta} \right) \right) \nonumber \\ 
& \le & G_{*} \left( \mathfrak{X}^{+} \left( x_{0}; \boldsymbol{\alpha} \right) \right) \sum_{\boldsymbol{\gamma}\, \in\, {\rm Cyl}_{\nu}}\ \sum_{\mathfrak{X}^{+} \left( y_{- \nu}; \boldsymbol{\gamma \beta} \right)\ \in\ \sigma^{- \nu} \left( \mathfrak{X}^{+} \left( y_{0}; \boldsymbol{\beta} \right) \right)} \hspace{-1cm} e^{ F^{*}_{\nu} \left( \mathfrak{X}^{+} \left( y_{- \nu}; \boldsymbol{\gamma \beta} \right) \right)} \nonumber \\ 
& = & G_{*} \left( \mathfrak{X}^{+} \left( x_{0}; \boldsymbol{\alpha} \right) \right). 
\end{eqnarray} 
Since the extreme quantities in the above inequality are one and the same, it is necessary that everything in between should also be equal to one another. 

We now state and prove a useful lemma regarding the density of backward images of points in $\mathscr{P}^{\mu, \Gamma} \left( \mathcal{X} \right)$ under the shift map $\sigma$. 

\begin{lemma} 
\label{density} 
The collection of backward images $\bigcup\limits_{\nu\, \ge\, 1} \left\{ \mathfrak{X}^{+} \left( y_{- \nu}; \boldsymbol{\gamma \beta} \right) : \boldsymbol{\gamma} \in {\rm Cyl}_{\nu} \right\}$ of any point $\mathfrak{X}^{+} \left( y_{0}; \boldsymbol{\beta} \right)$ in $\mathscr{P}^{\mu, \Gamma} \left( \mathcal{X} \right)$ is dense in $\mathscr{P}^{\mu, \Gamma} \left( \mathcal{X} \right)$. 
\end{lemma} 

\begin{proof} 
Recall that $\displaystyle{\bigcup\limits_{\nu\, \in\, \mathbb{Z}_{+}} \bigcup\limits_{\boldsymbol{\gamma}\, \in\, {\rm Cyl}_{\nu}} \left\{ \Pi_{\nu}^{-} \left( \mathfrak{X}_{\nu}^{-} \left( y_{0}; \boldsymbol{\gamma} \right) \right) \right\}}$ is a dense subset of $\mathcal{X}$ for any $y_{0}$ in $\mathcal{X}$, {\it i.e.}, given any $x_{0} \in \mathcal{X}$ and $\epsilon > 0$, there exists $\nu \in \mathbb{Z}_{+}$ such that $\left( y_{- \nu}, y_{0} \right) \in \Gamma^{\circ \nu}$ and $d_{\overline{\mathbb{C}}} \left( x_{0}, y_{- \nu} \right) < \epsilon$. Now consider the set of backward images of $\mathfrak{X}^{+} \left( y_{0}; \boldsymbol{\beta} \right) \in \mathscr{P}^{\mu, \Gamma} \left( \mathcal{X} \right)$ given by $\displaystyle{\bigcup\limits_{\nu\, \in\, \mathbb{Z}_{+}} \left\{ \mathfrak{X}^{+} \left( y_{- \nu}; \boldsymbol{\gamma \beta} \right) : \boldsymbol{\gamma} \in {\rm Cyl}_{\nu} \right\}}$. 

Given any point $\mathfrak{X}^{+} \left( x_{0}, \boldsymbol{\alpha} \right) \in \mathscr{P}^{\mu, \Gamma} \left( \mathcal{X} \right)$ and $\epsilon > 0$, we want to prove the existence of some $\nu \in \mathbb{Z}_{+}$ such that $d_{\mathscr{P}^{\mu, \Gamma} \left( \mathcal{X} \right)} \left( \mathfrak{X}^{+} \left( x_{0}, \boldsymbol{\alpha} \right), \mathfrak{X}^{+} \left( y_{- \nu}; \boldsymbol{\gamma \beta} \right) \right) < \epsilon$. 

For the given $\epsilon$, consider $n \in \mathbb{Z}_{+}$ such that $\dfrac{1}{2^{n}} < \epsilon$. We now look at $\Pi_{n}^{+} \left( \mathfrak{X}^{+} \left( x_{0}; \boldsymbol{\alpha} \right) \right) = x_{n} \in \mathcal{X}$ and mark the corresponding cylinder set $\boldsymbol{\alpha}' = \left( \alpha_{1} \cdots \alpha_{n} \right) \in {\rm Cyl}_{n}$. Then, there exists some $m \in \mathbb{Z}_{+}$ such that $d_{\overline{\mathbb{C}}} \left( x_{n}, y_{-m} \right) < \dfrac{\epsilon}{2^{n}}$, where $\left( y_{-m}, y_{0} \right) \in \Gamma^{\circ m}$ through some cylinder set $\boldsymbol{\gamma}  = \left( \gamma_{1} \cdots \gamma_{m} \right) \in {\rm Cyl}_{m}$. Since $\Gamma$ is expansive, we arrive at a point 
\[ \mathfrak{X}^{+} \left( y_{-(n + m)}; \boldsymbol{\alpha}' \boldsymbol{\gamma \beta} \right)\ =\ \left( y_{-(n + m)}, \cdots, y_{-m}, \cdots, y_{0}, \cdots; \alpha_{1} \cdots \alpha_{n} \gamma_{1} \cdots \gamma_{m} \beta_{1} \cdots \right)\ \in\ \mathscr{P}^{\mu, \Gamma} \left( \mathcal{X} \right), \] 
such that $d_{\mathscr{P}^{\mu, \Gamma} \left( \mathcal{X} \right)} \left( \mathfrak{X}^{+} \left( x_{0}, \boldsymbol{\alpha} \right), \mathfrak{X}^{+} \left( y_{-(n + m)}; \boldsymbol{\alpha}' \boldsymbol{\gamma \beta} \right) \right)\ \ <\ \ \epsilon$, thus proving the lemma. 
\end{proof} 

Finally, making use of the conclusion that we draw from Equation \eqref{usedensity} and Lemma \eqref{density}, we conclude that $G_{*}$ must be a constant function. Further, suppose there exists another sequence of positive integers $n_{k}$ such that $\mathcal{L}_{F^{*}}^{n_{k}} G$ converges to $G_{**}$, then $G_{**}$ is also a constant and satisfies $\sup \mathcal{L}_{F^{*}}^{n_{k}} G \to G_{**}$. However, $\left\{ \sup \mathcal{L}_{F^{*}}^{n_{k}} G \right\}_{k\, \ge\, 1}$ is a subsequence of $\left\{ \sup \mathcal{L}_{F^{*}}^{\nu} G \right\}_{\nu\, \ge\, 1}$, whose only limit point is $G_{*}$. Hence, $G_{*} \equiv G_{**}$. Thus, $\mathcal{L}_{F^{*}}^{\nu} G \to G_{*}$. 

\section{Proof of Claim \eqref{ucgce}} 
\label{proof} 

We have proved that the sequence $\left\{ \mathcal{L}_{F^{*}}^{\nu} G \right\}_{\nu\, \ge\, 1}$ converges to a constant, say $c(G)$, for every function $G \in \mathcal{C} \left( \mathscr{P}^{\mu, \Gamma} \left( \mathcal{X} \right), \mathbb{R} \right)$, where $F^{*} = F - \log \left( H \circ \sigma \right) + \log H - \log \Lambda$. 

Corresponding to the eigenfunction $H \in \mathcal{C} \left( \mathscr{P}^{\mu, \Gamma} \left( \mathcal{X} \right), \mathbb{R} \right)$, we now define a multiplication operator $M_{H}$ on $\mathcal{C} \left( \mathscr{P}^{\mu, \Gamma} \left( \mathcal{X} \right), \mathbb{R} \right)$ given by $M_{H} \left( G \right) = HG$ (pointwise multiplication). Then, from Equation \eqref{normalisedruelle}, we can write $\mathcal{L}_{F^{*}} = \dfrac{1}{\Lambda} M_{H^{-1}} \circ \mathcal{L}_{F} \circ M_{H}$. Thus, for any $\nu \in \mathbb{Z}_{+}$ and $G \in \mathcal{C} \left( \mathscr{P}^{\mu, \Gamma} \left( \mathcal{X} \right), \mathbb{R} \right)$, we have $\dfrac{1}{\Lambda^{\nu}} \mathcal{L}_{F}^{\nu} (G) = M_{H} \circ \mathcal{L}_{F^{*}}^{\nu} \left(\dfrac{G}{H}\right) \to c \left( \dfrac{G}{H} \right) H$. 

Considering $F = f \circ \Pi_{0}^{+},\ G = g \circ \Pi_{0}^{+}$ and $H = h \circ \Pi_{0}^{+}$, we now obtain that $\dfrac{1}{\Lambda^{\nu}} \mathcal{L}_{f}^{\nu} (g) \to c \left( \dfrac{g}{h} \right) h$. Thus, the proof of Claim \eqref{ucgce} and thereby, Theorem \eqref{rot} is complete. 

\begin{corollary} 
For any $f \in \mathcal{C}^{\alpha} \left( \mathcal{X}, \mathbb{R} \right)$, the eigenvalue $\Lambda$ of the Ruelle operator $\mathcal{L}_{f} : \mathcal{C}^{\alpha} \left( \mathcal{X}, \mathbb{R} \right) \longrightarrow \mathcal{C}^{\alpha} \left( \mathcal{X}, \mathbb{R} \right)$, is equal to $e^{{\rm Pr} (\Gamma\vert_{\mathcal{X} \times \mathcal{X}}, f)}$. 
\end{corollary} 

\begin{proof} 
As a particular case of the convergence that we have proved in Claim \eqref{ucgce} of Theorem \eqref{rot} , we have 
\begin{equation} 
\label{cgstopr} 
\frac{1}{\nu} \log \mathcal{L}_{f}^{\nu} \left(1\right)\ \ \to\ \ \log \Lambda,\ \ \ \ \text{as}\ \ \nu \to \infty. 
\end{equation} 

For any fixed $x_{0} \in \mathcal{X}$, there exists $\epsilon_{x_{0}} > 0$ such that whenever there exists distinct pre-images, say $x_{-1}$ and $x_{-1}'$ through the same variety, then they are at least $\epsilon_{x_{0}}$-apart from each other, {\it i.e.}, $d_{\overline{\mathbb{C}}} \left(x_{-1}, x_{-1}'\right) > \epsilon_{x_{0}}$ for all points $\left( x_{-1}, x_{0} \right) \in \Gamma_{i}$ and $\left( x_{-1}', x_{0} \right) \in \Gamma_{i}$ for $1 \le i \le N$. Extending these backward orbits of length of $1$ to obtain backward orbits of length $\nu$ that land at $x_{0}$ through the same combinatorial variety, we get a family of $(\nu, \epsilon_{x_{0}})$-separated orbits, namely $\left\{ \mathfrak{X}_{\nu}^{-} \left( x_{0}; \boldsymbol{\alpha} \right) : \boldsymbol{\alpha} \in {\rm Cyl}_{\nu} \right\}$ for all $\nu \in \mathbb{Z}_{+}$. 
\begin{equation} 
\label{lepr} 
\lim_{\nu\, \to\, \infty} \frac{1}{\nu} \log \left( \sum_{\boldsymbol{\alpha}\, \in\, {\rm Cyl}_{\nu}} \sum_{\mathfrak{X}_{\nu}^{-} \left( x_{0}; \boldsymbol{\alpha} \right)} e^{\left( \mathcal{Z}_{\nu}^{+} (f) \right) \left( \mathfrak{X}_{\nu}^{-} \left( x_{0}; \boldsymbol{\alpha} \right) \right)} \right)\ \ =\ \ \lim_{\nu\, \to\, \infty} \frac{1}{\nu} \log \left( \mathcal{L}_{f}^{\nu} (1) (x_{0}) \right)\ \ \le\ \ {\rm Pr} (\Gamma\vert_{\mathcal{X} \times \mathcal{X}}, f). 
\end{equation} 

Also, for every $x_{0} \in \mathcal{X}$ and $\epsilon > 0$, there exists $N_{x_{0}} \in \mathbb{Z}_{+}$ such that for any $y_{0} \in \mathcal{X}$ there exists a pre-image of $x_{0}$ given by $\left( x_{- N_{x_{0}}}, x_{0} \right) \in \Gamma^{\circ N_{x_{0}}}$ that satisfies $d_{\overline{\mathbb{C}}} \left(x_{- N_{x_{0}}}, y_{0}\right) < \epsilon$. Then, for any $\nu$-long orbit $\mathfrak{X}_{\nu}^{-} \left( y_{0}; \boldsymbol{\beta} \right) \in \mathscr{Q}_{\nu}^{\Gamma} (\mathcal{X})$, we can find another $\nu$-long orbit $\mathfrak{X}_{\nu}^{-} \left(x_{- N_{x_{0}}}; \boldsymbol{\beta} \right)$ so that the points are atmost $\epsilon$-apart. Thus, the collection of points $\left\{ \mathfrak{X}_{\nu}^{-} \left(x_{- N_{x_{0}}}; \boldsymbol{\beta} \right) : \boldsymbol{\beta} \in {\rm Cyl}_{\nu} \right\}$ forms a $(\nu, \epsilon)$-spanning set of $\mathscr{Q}_{\nu}^{\Gamma} (\mathcal{X})$. Hence, 
\begin{eqnarray} 
\label{gepr} 
\lim_{\nu\, \to\, \infty} \frac{1}{\nu} \log \left( \sum_{\boldsymbol{\beta}\, \in\, {\rm Cyl}_{\nu}} \sum_{\mathfrak{X}_{\nu}^{-} \left( x_{- N_{x_{0}}}; \boldsymbol{\beta} \right)} e^{\left( \mathcal{Z}_{\nu}^{+} (f) \right) \left( \mathfrak{X}_{\nu}^{-} \left( x_{- N_{x_{0}}}; \boldsymbol{\beta} \right) \right)} \right) & = & \lim_{\nu\, \to\, \infty} \frac{1}{\nu} \log \left( \mathcal{L}_{f}^{\nu} (1) \right) (x_{0}) \nonumber \\ 
& \ge & {\rm Pr} (\Gamma\vert_{\mathcal{X} \times \mathcal{X}}, f). 
\end{eqnarray} 

Thus, from Equations \eqref{cgstopr}, \eqref{lepr} and \eqref{gepr}, we obtain \[ \lim\limits_{\nu\, \to\, \infty} \dfrac{1}{\nu} \log \left( \mathcal{L}_{f}^{\nu} (1) \right) (x_{0})\ \ =\ \ {\rm Pr} (\Gamma\vert_{\mathcal{X} \times \mathcal{X}}, f)\ \ =\ \ \log \Lambda,\ \ \ \text{for any point}\ x_{0} \in \mathcal{X}. \]  
\end{proof} 

\begin{corollary} 
$\Lambda$ is the maximal eigenvalue of $\mathcal{L}_{f}$, with the remainder of the spectrum of $\mathcal{L}_{f}$ lying in a disc of radius strictly smaller than $\Lambda$, whenever $f$ is a H\"{o}lder continuous function and the domain of $\mathcal{L}_{f}$ is restricted to the same space. 
\end{corollary} 

\begin{proof} 
Let $f \in \mathcal{H}^{s} \left( \mathcal{X}, \mathbb{R} \right) \subset \mathcal{C}^{\alpha} \left( \mathcal{X}, \mathbb{R} \right)$ where $\mathcal{H}^{s} \left( \mathcal{X}, \mathbb{R} \right)$ denotes the Banach space of H\"{o}lder continuous functions with H\"{o}lder exponent $s$. Corresponding to $f$, we know that we have $F = f \circ \Pi_{0}^{+} \in \mathcal{H}^{s} \left( \mathscr{P}^{\mu, \Gamma} \left( \mathcal{X} \right), \mathbb{R} \right)$. Then, we know that the eigenfunction $h$ of $\mathcal{L}_{f}$ and the corresponding eigenfunction $H = h \circ \Pi_{0}^{+}$ of $\mathcal{L}_{F}$ also belong to $\mathcal{H}^{s} \left( \mathcal{X}, \mathbb{R} \right)$ and $\mathcal{H}^{s} \left( \mathscr{P}^{\mu, \Gamma} \left( \mathcal{X} \right), \mathbb{R} \right)$ respectively. Thus, $F^{*}$, as defined in Equation \eqref{normalF} belongs to $\mathcal{H}^{s} \left( \mathscr{P}^{\mu, \Gamma} \left( \mathcal{X} \right), \mathbb{R} \right)$. 

Denoting by $\mathcal{M} \left( \mathscr{P}^{\mu, \Gamma} \left( \mathcal{X} \right), \sigma \right)$, the weak-* compact space of all $\sigma$-invariant probability measures supported on $\mathscr{P}^{\mu, \Gamma} \left( \mathcal{X} \right)$, one can prove that the adjoint operator $\left( \mathcal{L}_{F^{*}} \right)^{*}$ restricted on $\mathcal{M} \left( \mathscr{P}^{\mu, \Gamma} \left( \mathcal{X} \right), \sigma \right)$ has a fixed point, say $M$ using the Schauder-Tychonoff theorem. Further, since we know from the proof of Claim \eqref{ucgce} that the sequence $\mathcal{L}_{F^{*}}^{\nu} \left(G\right) \to c \left(G\right)$, we conclude that 
\[ \mathcal{L}_{F^{*}}^{\nu} \left(G\right) \to \int G \mathrm{d} M. \] 

Now writing $\mathcal{H}^{s} \left( \mathscr{P}^{\mu, \Gamma} \left( \mathcal{X} \right), \mathbb{R} \right) = W_{1} \oplus W_{2}$ where 
\begin{eqnarray*} 
W_{1} & = & \left\{ G \in \mathcal{H}^{s} \left( \mathscr{P}^{\mu, \Gamma} \left( \mathcal{X} \right), \mathbb{R} \right)\ :\ \int G \mathrm{d} M = 0 \right\}\ \ \text{and} \\ 
W_{2} & = & \text{the set of all constant functions}, 
\end{eqnarray*} 
and following through the arguments in Theorem (4.5) in \cite{pp:1990}, we obtain $1$ as the maximal eigenvalue of $\mathcal{L}_{F^{*}}$ with the remainder of the spectrum lying in a disc of radius strictly smaller than $1$. Thus, $\Lambda$ is the maximal eigenvalue of $\mathcal{L}_{F}$ with the remainder of the spectrum of $\mathcal{L}_{F}$ lying in a disc of radius strictly smaller than $\Lambda$. The case for $\mathcal{L}_{f}$ follows suit. 
\end{proof}

\vspace{+2cm} 

{\bf Authors' Affiliations and Contact coordinates}: 
\bigskip 

{\sc Shrihari Sridharan} \\ 
Indian Institute of Science Education and Research Thiruvananthapuram (IISER-TVM), \\ 
Maruthamala P.O., Vithura, Kerala, India. \\ 
\texttt{shrihari@iisertvm.ac.in} \\ 
\bigskip 

{\sc Subith G.} \\ 
Indian Institute of Science Education and Research Thiruvananthapuram (IISER-TVM), \\ 
Maruthamala P.O., Vithura, Kerala, India. \\ 
\texttt{subith21@iisertvm.ac.in} 

\end{document}